\documentclass[12pt, letterpaper]{amsart}
\usepackage[utf8]{inputenc}
\usepackage[margin=1in]{geometry}

\usepackage{amssymb}
\usepackage{amsfonts}
\usepackage{amsmath}
\usepackage{lmodern}
\usepackage{microtype}
\usepackage{mathtools}
\usepackage[dvipsnames]{xcolor}
\usepackage[
    colorlinks,
    linkcolor=blue,
    citecolor=blue,
    urlcolor=blue,
    implicit=true,
    breaklinks=true
]{hyperref}
\usepackage[noabbrev,capitalize]{cleveref}
\usepackage[shortlabels]{enumitem}

\newcommand{\R}{\mathbb{R}}

\newcommand{\Efrak}{\mathfrak{E}}

\newcommand{\Ucal}{\mathcal{U}}
\newcommand{\Pcal}{\mathcal{P}}
\newcommand{\Pcalp}[1]{\mathcal{P}_{#1}}
\newcommand{\PcalW}{\mathcal{P}_2}

\DeclareMathOperator{\proj}{proj}
\newcommand{\dd}{\mathrm{d}}
\newcommand{\Tr}{\mathrm{Tr}}

\DeclarePairedDelimiter{\norm}{\lVert}{\rVert}
\DeclarePairedDelimiter{\abs}{\lvert}{\rvert}
\DeclarePairedDelimiterX{\scpr}[2]{\langle}{\rangle}{#1, #2} 

\DeclareMathOperator*{\argmax}{arg\,max}
\DeclareMathOperator*{\Law}{Law}

\theoremstyle{plain}
\newtheorem{theorem}{Theorem}[section]
\newtheorem{proposition}[theorem]{Proposition}

\theoremstyle{definition}
\newtheorem{definition}[theorem]{Definition}
\newtheorem{assumption}[theorem]{Assumption}

\theoremstyle{remark}
\newtheorem{remark}[theorem]{Remark}

\numberwithin{equation}{section} 

\begin{document}

\title[Mean-Field Theory of $\Theta$-Expectations]{A Mean-Field Theory of $\Theta$-Expectations}

\author{Qian Qi}
\address{Peking University, Beijing 100871, China}
\email{qiqian@pku.edu.cn}
\date{\today}

\begin{abstract}
The canonical theory of sublinear expectations, a foundation of stochastic calculus under ambiguity, is insensitive to the non-convex geometry of primitive uncertainty models. This paper develops a new stochastic calculus for a structured class of such non-convex models. We introduce a class of fully coupled Mean-Field Forward-Backward Stochastic Differential Equations where the BSDE driver is defined by a pointwise maximization over a law-dependent, non-convex set. Mathematical tractability is achieved via a uniform strong concavity assumption on the driver with respect to the control variable, which ensures the optimization admits a unique and stable solution. A central contribution is to establish the Lipschitz stability of this optimizer from primitive geometric and regularity conditions, which underpins the entire well-posedness theory. We prove local and global well-posedness theorems for the FBSDE system. The resulting valuation functional, the $\Theta$-Expectation, is shown to be dynamically consistent and, most critically, to violate the axiom of sub-additivity. This, along with its failure to be translation invariant, demonstrates its fundamental departure from the convex paradigm. This work provides a rigorous foundation for stochastic calculus under a class of non-convex, endogenous ambiguity.
\end{abstract}

\maketitle
\tableofcontents

\section{Introduction}

\subsection{The Identifiability Impasse of Convex Models of Ambiguity}
The modern mathematical theory of stochastic calculus under ambiguity is dominated by the seminal framework of sublinear expectations and the associated G-calculus, pioneered by \cite{Peng2019}. This theory furnishes a robust and mathematically elegant generalization of classical Itô calculus by replacing the linear expectation operator with a functional satisfying sub-additivity and positive homogeneity. The Daniell-Stone representation theorem, a foundation of convex analysis, guarantees that any such functional admits a dual representation as the supremum of linear expectations taken over a weakly compact and \emph{convex} set of probability measures, $\Pcal$ (see, e.g., \cite{Rockafellar1970}).

This foundation, while mathematically powerful, imposes a profound structural limitation which we term the \textbf{identifiability impasse}. The valuation functional, being the support function of the set of priors $\Pcal$, is intrinsically insensitive to the underlying geometry of the uncertainty set; it depends only on its convex hull, $\text{conv}(\Pcal)$. Consequently, if the primitive set of plausible models, derived perhaps from statistical inference or first principles, is inherently non-convex, the sublinear framework cannot distinguish it from its convexification. This is a critical deficiency in applications where the specific topological and geometric features of the uncertainty set, such as bi-modal beliefs, disjoint scenario clusters, or other holes in the model space, are of first-order importance. The very structure of the agent's uncertainty is erased by the valuation mechanism.

\subsection{A Tractable Calculus for a Class of Non-Convex Ambiguity Models}
This paper develops a new mathematical paradigm that circumvents this identifiability impasse for a broad and structured class of non-convex ambiguity models. We depart from the axiom of sub-additivity and construct a valuation theory directly from the primitive, non-convex set of models. A related paper, the $\theta$-expectation (see \cite{qi2025theorythetaexpectations}), demonstrates a theory for a non-convex yet exogenously given uncertain set. Beyond \cite{qi2025theorythetaexpectations}, this paper is founded on a novel class of Forward-Backward Stochastic Differential Equations (FBSDEs) to endogenize the non-convex uncertain set, which we term \textbf{Mean-Field $\Theta$-FBSDEs} (see \cref{def:mf_theta_fbsde}).

This framework is characterized by two principal innovations. First, the driver of our BSDE is determined at each instant $(t,\omega)$ through a direct, pointwise maximization over the primitive uncertainty set $\Ucal$. The control variable, denoted by $a$, represents the agent's choice of a local dynamic model, and the resulting valuation is thus linked by construction to the full geometry of $\Ucal$. Second, we introduce a sophisticated mean-field interaction whereby the uncertainty set itself, denoted $\Ucal_\theta$, depends on a parameter $\theta = g(\mu_t)$ determined by the law $\mu_t$ of the solution's backward component (the value process). This renders the ambiguity \emph{endogenous}, providing a natural and powerful mechanism for modeling systemic risk, where the aggregate state of the system shapes individual perceptions of uncertainty.

The central mathematical obstacle is that optimization over a non-convex set is generically ill-posed and unstable. The crucial innovation of this work is the identification of a structural condition that renders the problem tractable: a \textbf{uniform strong concavity assumption} on the driver $F(t,x,y,z,a,\mu)$ with respect to the control variable $a$ (\cref{ass:mf_fbsde_standing}(iv)). This is a significant modeling choice. It represents a fundamental trade-off: we relax the convexity of the domain of optimization (the set $\Ucal$) but impose a convexity-like condition on the objective function itself. This implies that our framework does not generalize theories like Peng's G-calculus, whose canonical driver is linear in the control. Rather, our theory carves out a different, complementary class of problems, such as those involving quadratic regularization, risk-penalties, or utility maximization, where strong concavity arises naturally. This assumption guarantees that the optimization problem admits a \emph{unique} maximizer, which we denote $a^*$.

The existence of a unique maximizer is insufficient; one must establish its stable dependence on the state variables to control the system's dynamics. A principal contribution of this paper is to prove, under a set of verifiable primitive conditions on the geometry of the uncertainty sets and the regularity of the constraints (\cref{ass:boundary_and_ndg}), that the resulting optimizer map $a^*$ is uniformly Lipschitz continuous (\cref{prop:lipschitz_stability}). This result is the foundation of our well-posedness theory. The proof is a delicate application of sensitivity analysis for parametric nonlinear programs, relying on a quantitative version of the implicit function theorem applied to the Karush-Kuhn-Tucker (KKT) system. This fundamental stability property ensures that the driver, obtained by substituting the optimizer back into the primitive driver, inherits the Lipschitz regularity required for modern FBSDE theory.

\subsection{Main Contributions and Outline of the Paper}
With these foundational results in place, we construct a complete theory for the Mean-Field $\Theta$-FBSDE system. Our main contributions are as follows:

\begin{enumerate}
    \item \textbf{A New Class of FBSDEs.} We introduce and formalize the Mean-Field $\Theta$-FBSDE (\cref{def:mf_theta_fbsde}), a fully coupled system where the dynamics of a forward process and a backward valuation are intertwined. The system's evolution is driven by a pointwise optimization over a law-dependent, non-convex set, providing a novel model for endogenous and non-convex ambiguity.

    \item \textbf{Rigorous Well-Posedness Results.} We establish local-in-time existence and uniqueness of solutions (\cref{thm:mf_fbsde_local_wellposedness}). The proof proceeds via a contraction mapping argument on a carefully chosen product space of stochastic processes, where the Lipschitz stability of the optimizer map (\cref{prop:lipschitz_stability}) is the essential ingredient. We then prove a global-in-time well-posedness result (\cref{thm:global_wellposedness}) under an additional, but standard, strong monotonicity assumption on the system's coefficients.

    \item \textbf{A Dynamically Consistent Calculus Beyond Convexity.} We define the \textbf{$\Theta$-Expectation} as the valuation functional generated by the unique solution to our BSDE. We prove it satisfies the crucial properties of dynamic consistency and monotonicity. Most importantly, we provide a rigorous proof that it fails to be sub-additive (\cref{prop:theta_exp_properties}), the defining characteristic of the convex paradigm. A sharp counterexample, which also demonstrates the failure of translation invariance, shows the driver can be locally convex, definitively demonstrating the departure of our theory from convex frameworks and its ability to resolve the identifiability impasse.

    \item \textbf{Connections to Non-Linear PDEs.} We establish a nonlinear Feynman-Kac representation (\cref{sec:feynman_kac}), connecting our FBSDE to a class of highly non-local, Hamilton-Jacobi-Bellman-type PDEs (\cref{eq:theta_hjb_pde}). Acknowledging the analytical difficulty of such equations, we frame the result in the modern language of \emph{viscosity solutions} (\cref{thm:feynman_kac_detailed}). The paper also provides a formal derivation of the associated \textbf{Master Equation} on the Wasserstein space (\cref{sec:master_equation}), highlighting the deep structural novelty of the framework. We explicitly note that a rigorous analysis of this infinite-dimensional PDE is a formidable open problem and a direction for future research.
\end{enumerate}

This work provides a rigorous foundation for stochastic calculus under a class of non-convex and endogenous ambiguity models. It demonstrates that it is possible to build a tractable and identifiable theory where the fine-grained geometry of uncertainty is of first-order importance, thereby opening new avenues for modeling complex systems in physics, finance, economics, and control theory.

\section{Preliminaries: Function Spaces and the Wasserstein Space}

We begin by establishing the mathematical setting with precision. Let $T>0$ be a finite and fixed time horizon. We work on a complete probability space $(\Omega, \mathcal{F}, P_0)$, which is assumed to be rich enough to support a standard $d$-dimensional Brownian motion $B = (B_t)_{t \in [0,T]}$. The filtration $\mathbb{F}=(\mathcal{F}_t)_{t \in [0,T]}$ is taken to be the $P_0$-completion of the natural filtration generated by $B$, thereby satisfying the usual conditions of right-continuity and completeness. We denote by $\mathbb{E}_{P_0}$ the expectation with respect to the reference measure $P_0$. All notions of almost sure ($P_0$-a.s.) or almost everywhere (a.e.) equality and inequality are with respect to $P_0$ or the product measure $dt \otimes dP_0$ on $[0,T]\times\Omega$, respectively.

\subsection{Spaces of Stochastic Processes}
The analytical framework for Forward-Backward Stochastic Differential Equations is built upon specific Banach spaces of stochastic processes. The norms on these spaces are designed to provide control over both the temporal and stochastic fluctuations of the solution paths.

\begin{definition}[Spaces of Processes]\label{def:process_spaces}
For a given real number $p \ge 1$ and a separable Hilbert space $(H, \norm{\cdot}_H)$, we define the following spaces of $\mathbb{F}$-adapted processes:
\begin{itemize}
    \item The space $\mathcal{S}^p([0,T]; H)$ consists of $H$-valued, continuous processes $X=(X_t)_{t\in[0,T]}$ such that
    \[ \norm{X}_{\mathcal{S}^p} := \left(\mathbb{E}_{P_0}\left[\sup_{t \in [0,T]} \norm{X_t}_H^p\right]\right)^{1/p} < \infty. \]
    \item The space $\mathcal{H}^p([0,T]; H)$ consists of $H$-valued, $\mathbb{F}$-predictable processes $Z=(Z_t)_{t\in[0,T]}$ such that
    \[ \norm{Z}_{\mathcal{H}^p} := \left(\mathbb{E}_{P_0}\left[\left(\int_0^T \norm{Z_s}_H^2 ds\right)^{p/2}\right]\right)^{1/p} < \infty. \]
    \item For a separable Banach space $(E_A, \norm{\cdot}_{E_A})$, the space $\mathcal{L}^p([0,T]; E_A)$ consists of $E_A$-valued, $\mathbb{F}$-predictable processes $A=(A_t)_{t\in[0,T]}$ such that
    \[ \norm{A}_{\mathcal{L}^p} := \left(\mathbb{E}_{P_0}\left[\int_0^T \norm{A_s}_{E_A}^p ds\right]\right)^{1/p} < \infty. \]
\end{itemize}
When the context makes the domain $[0,T]$ or the space $H$ (typically $\R^k$) clear, we may abbreviate the notation (e.g., $\mathcal{S}^p$).
\end{definition}

\begin{proposition}\label{prop:banach_spaces}
The spaces $\mathcal{S}^p([0,T]; \R^k)$ and $\mathcal{H}^p([0,T]; \R^k)$ are real Banach spaces for any $p \ge 1$ and $k \ge 1$.
\end{proposition}
\begin{proof}
This is a foundational result in stochastic analysis, whose proof we include for completeness as it relies on foundations of the theory. Let $H=\R^k$.

\textbf{1. Completeness of $\mathcal{H}^p([0,T]; H)$.}
The space $\mathcal{H}^p$ is isometric to the Bochner space $L^p(\Omega; L^2([0,T]; H))$, where a process $Z$ is identified with the random variable $\omega \mapsto (t \mapsto Z_t(\omega))$. The completeness of Bochner spaces is a classical result of functional analysis. Thus, $\mathcal{H}^p$ is a Banach space.

\textbf{2. Completeness of $\mathcal{S}^p([0,T]; H)$.}
This is more subtle. Let $(X^n)_{n\ge 1}$ be a Cauchy sequence in $\mathcal{S}^p$.
The norm is $\norm{X}_{\mathcal{S}^p} = \norm*{\sup_{t \in [0,T]} \norm{X_t}_H}_{L^p(\Omega)}$. Since $L^p(\Omega)$ is complete, the sequence of real-valued random variables $M^n := \sup_{t \in [0,T]} \norm{X^n_t}_H$ converges in $L^p(\Omega)$ to some $M \in L^p(\Omega)$.
For any fixed $t \in [0,T]$, $\norm{X^n_t - X^m_t}_H \le \sup_{s \in [0,T]} \norm{X^n_s - X^m_s}_H$. Taking the $L^p$ norm of both sides shows that $(X^n_t)_{n\ge 1}$ is a Cauchy sequence in $L^p(\Omega; H)$. By completeness of $L^p(\Omega;H)$, there exists a limit process $X_t$ for each $t$.

We must show that this limit process $X$ has a continuous version which is the limit of $(X^n)$ in the $\mathcal{S}^p$ norm.
By the triangle inequality for the $\mathcal{S}^p$ norm,
\[ \left\| \sup_t \norm{X_t^n}_H \right\|_{L^p} \le \left\| \sup_t \norm{X_t^m}_H \right\|_{L^p} + \left\| \sup_t \norm{X_t^n - X_t^m}_H \right\|_{L^p}. \]
Since $(X^n)$ is Cauchy, the sequence of norms is bounded.
Let $\tau_{n,m}(\epsilon) := \inf\{t \in [0,T] : \norm{X^n_t - X^m_t}_H > \epsilon\}$ be a stopping time. By Chebyshev's inequality and the Cauchy property of $(X^n)$:
\[ P_0(\sup_t \norm{X^n_t - X^m_t}_H > \epsilon) \le \frac{1}{\epsilon^p} \mathbb{E}_{P_0}\left[\sup_t \norm{X^n_t - X^m_t}_H^p\right] \to 0 \quad \text{as } n,m \to \infty. \]
This shows that $X^n$ converges uniformly on $[0,T]$ in probability. A standard result states that this implies the existence of a subsequence $(X^{n_k})$ and a continuous process $X$ such that $X^{n_k}_t \to X_t$ a.s. for all $t$ and $\sup_t \norm{X^{n_k}_t - X_t}_H \to 0$ a.s.

Now we show convergence of the full sequence in $\mathcal{S}^p$. For the subsequence $(X^{n_k})$, we have $\sup_t \norm{X^{n_k}_t - X_t}_H \to 0$ a.s. By the Cauchy property, for any $\epsilon > 0$, there exists $N$ such that for $n_k, m > N$, $\norm{X^{n_k} - X^m}_{\mathcal{S}^p} < \epsilon$. Applying Fatou's lemma to the variable $\omega$ as $m\to\infty$ along the convergent subsequence gives
\[ \mathbb{E}_{P_0}\left[ \liminf_{m\to\infty} \sup_t \norm{X^{n_k}_t - X^m_t}_H^p \right] \le \liminf_{m\to\infty} \mathbb{E}_{P_0}\left[ \sup_t \norm{X^{n_k}_t - X^m_t}_H^p \right] \le \epsilon^p. \]
Since $\sup_t \norm{X^{n_k}_t - X_t}_H \le \liminf_{m\to\infty} \sup_t \norm{X^{n_k}_t - X^m_t}_H$ a.s., we have $\norm{X^{n_k} - X}_{\mathcal{S}^p} \le \epsilon$.
This shows that the subsequence $(X^{n_k})$ converges to $X$ in $\mathcal{S}^p$. Since the original sequence $(X^n)$ is Cauchy and has a convergent subsequence, the entire sequence must converge to the same limit $X$. The limit process $X$ is continuous and adapted, hence $X \in \mathcal{S}^p$. This establishes completeness.
\end{proof}

\subsection{The Wasserstein Space of Probability Measures}
Our mean-field framework requires a topological and metric structure on the space of probability measures. The Wasserstein distance provides the natural geometry for problems involving optimal transport.

\begin{definition}[Wasserstein Space] \label{def:wasserstein}
Let $p \ge 1$ and let $(E, d_E)$ be a Polish space. We denote by $\Pcalp{p}(E)$ the set of Borel probability measures $\mu$ on $E$ with a finite $p$-th moment, i.e., $\int_{E} d_E(x_0, x)^p \dd\mu(x) < \infty$ for some (and hence for any) $x_0 \in E$. For $\mu, \nu \in \Pcalp{p}(E)$, the \textbf{$p$-Wasserstein distance} is defined as
\begin{equation} \label{eq:wasserstein_def}
 W_p(\mu, \nu) := \left( \inf_{\pi \in \Pi(\mu, \nu)} \int_{E \times E} d_E(x, y)^p \dd\pi(x, y) \right)^{1/p},
\end{equation}
where $\Pi(\mu, \nu)$ is the set of all joint probability measures (couplings) on the product space $E \times E$ having $\mu$ and $\nu$ as their first and second marginals, respectively.
\end{definition}

In this paper, we focus on $E=\R^k$ with the Euclidean metric, and we primarily use the case $p=2$, for which we use the shorthand $\PcalW(\R^k) \equiv \Pcalp{2}(\R^k)$. The space $(\PcalW(\R^k), W_2)$ possesses a rich geometric structure.

\begin{proposition}\label{prop:polish_space}
For $k \ge 1$ and $p \ge 1$, the metric space $(\Pcalp{p}(\R^k), W_p)$ is a Polish space (a complete, separable metric space). Furthermore, convergence in $W_p$ is equivalent to weak convergence of measures plus convergence of the $p$-th moments.
\end{proposition}
\begin{proof}
This is a classical result in optimal transport theory. A full proof is substantial and can be found in foundational texts such as \cite[Theorem 6.18]{Villani2009} or \cite[Theorem 5.10]{Ambrosio2008}.
\end{proof}

\subsection{Differentiability on Wasserstein Space}
The formal derivation of the Master Equation in Section \ref{sec:master_equation} requires a notion of differentiation for functions defined on the Wasserstein space. The standard concept, which lifts Fréchet differentiation from an $L^2$ space, was pioneered by P.-L. Lions in his lectures at the Collège de France and systematically developed by Cardaliaguet, Carmona, Delarue, and others.

\begin{definition}[Lions Derivative] \label{def:lions_derivative}
A function $U: \PcalW(\R^k) \to \R$ is said to be \textbf{L-differentiable} at a measure $\mu \in \PcalW(\R^k)$ if the following holds:
For any probability space $(\Omega, \mathcal{F}, P)$ and any random variable $X \in L^2(\Omega, \mathcal{F}, P; \R^k)$ with $\Law(X) = \mu$, the lifted map $\tilde{U}: L^2(\Omega, \mathcal{F}, P; \R^k) \to \R$, defined by $\tilde{U}(Y) := U(\Law(Y))$, is Fréchet differentiable at $X$.

The Fréchet derivative $D\tilde{U}(X)$ is a continuous linear functional on $L^2(\Omega, \mathcal{F}, P; \R^k)$. By the Riesz representation theorem, there exists a unique random variable $\xi \in L^2(\Omega, \mathcal{F}, P; \R^k)$ such that for all $Y \in L^2(\Omega, \mathcal{F}, P; \R^k)$,
\[ D\tilde{U}(X)[Y] = \mathbb{E}\left[ \scpr{\xi}{Y} \right]. \]
A fundamental result of the theory is that this representing random variable $\xi$ is necessarily measurable with respect to the $\sigma$-algebra generated by $X$, i.e., $\xi \in L^2(\Omega, \sigma(X), P; \R^k)$. By the Doob-Dynkin lemma, there exists a Borel-measurable function $\partial_\mu U(\mu, \cdot): \R^k \to \R^k$ such that $\xi = \partial_\mu U(\mu, X)$, $P$-a.s. This function $\partial_\mu U(\mu, \cdot)$, which is uniquely defined up to a $\mu$-null set, is called the \textbf{Lions derivative} (or L-derivative) of $U$ at $\mu$.
\end{definition}

\begin{remark}[Characterization and Independence]\label{rem:lions_derivative_characterization}
The definition is robust in that the existence of the derivative and the resulting function $\partial_\mu U(\mu, \cdot)$ are independent of the choice of probability space or the specific random variable $X$ used for the lift. If $U$ is L-differentiable, then for any $X$ with $\Law(X)=\mu$ and any $Y \in L^2(\Omega, \sigma(X), P; \R^k)$, the directional derivative of $U$ can be expressed as:
\[ \lim_{h\to 0} \frac{U(\Law(X+hY)) - U(\Law(X))}{h} = \mathbb{E}_{P}\left[ \scpr{\partial_\mu U(\mu, X)}{Y} \right]. \]
The function $\partial_\mu U(\mu, \cdot)$ can be interpreted as the gradient of $U$ in the direction of a change in the measure $\mu$ at the point $x$. A comprehensive treatment of this theory, including the proofs of the key structural results, can be found in the notes from \cite{Cardaliaguet2013} or the two-volume monograph by \cite{CarmonaDelarue2018}.
\end{remark}

\section{The Mean-Field \texorpdfstring{$\Theta$}{Theta}-FBSDE System}
\label{sec:theta_fbsde_system}

This section introduces the central mathematical object of our study: a fully coupled Forward-Backward Stochastic Differential Equation system where ambiguity enters through a pointwise optimization over a law-dependent, non-convex set. This structure is designed to capture endogenous uncertainty, where the agents' collective behavior, reflected in the law of the value process, dynamically shapes the set of perceived plausible models.

Let $(E_A, \norm{\cdot}_{E_A})$ be a separable Banach space, which will serve as the space of control variables. A control variable $a \in E_A$ represents a specific model choice for the local dynamics, for instance, a drift or volatility parameter.

\begin{definition}[Endogenous Uncertainty Set]
\label{def:endogenous_uncertainty_set}
Let $\Theta$ be a compact metric space representing the possible regimes of systemic ambiguity. Let $g: \Pcal_2(\R) \to \Theta$ be a continuous map. For each $\theta \in \Theta$, let $\Ucal_\theta$ be a non-empty, compact subset of the control space $E_A$. We refer to the law-dependent set $\Ucal_{g(\mu)} \subset E_A$ as the \textbf{endogenous uncertainty set}. The map $g$ translates the state of the system (the law of the value process $\mu$) into a specific ambiguity regime $\theta$, which in turn selects the set of available models $\Ucal_\theta$.
\end{definition}

\begin{definition}[The Mean-Field $\Theta$-FBSDE]
\label{def:mf_theta_fbsde}
Let $x_0 \in \R^k$ be the initial state and $\Phi: \R^k \to \R$ be the terminal cost function. Let the coefficients be functions $b: [0,T] \times \Omega \times \R^k \times E_A \times \Pcal_2(\R) \to \R^k$, $\sigma: [0,T] \times \Omega \times \R^k \times E_A \times \Pcal_2(\R) \to \R^{k \times d}$, and $F: [0,T] \times \Omega \times \R^k \times \R \times \R^{d} \times E_A \times \Pcal_2(\R) \to \R$. A solution to the Mean-Field $\Theta$-FBSDE is a quadruple of $\mathbb{F}$-adapted processes $(X, Y, Z, A)$ valued in $\R^k \times \R \times \R^{d} \times E_A$ that satisfies, $P_0$-a.s., the following coupled system:
\begin{enumerate}
    \item \textbf{(Forward Evolution):} For all $t \in [0,T]$,
    \begin{equation} \label{eq:mf_fbsde_forward}
        X_t = x_0 + \int_0^t b(s, X_s, A_s, \mu_s) \dd s + \int_0^t \sigma(s, X_s, A_s, \mu_s) \dd B_s.
    \end{equation}
    
    \item \textbf{(Backward Evolution):} For all $t \in [0,T]$,
    \begin{equation} \label{eq:mf_fbsde_backward}
        Y_t = \Phi(X_T) + \int_t^T F(s, X_s, Y_s, Z_s, A_s, \mu_s) \dd s - \int_t^T Z_s \dd B_s.
    \end{equation}
    In both equations, $\mu_s := \Law_{P_0}(Y_s)$ is the law of the value process $Y$ at time $s$.

    \item \textbf{(Pointwise Optimality):} For almost every $(t,\omega) \in [0,T] \times \Omega$,
    \begin{equation} \label{eq:mf_fbsde_optimality}
        A_t(\omega) \in \underset{a \in \Ucal_{g(\mu_t)}}{\argmax} \; F(t, \omega, X_t(\omega), Y_t(\omega), Z_t(\omega), a, \mu_t).
    \end{equation}
\end{enumerate}
\end{definition}

\begin{remark}[On the Nature of the Coupling]
The system defined above exhibits multiple, deeply intertwined layers of coupling, which represent the primary source of its mathematical complexity.
\begin{enumerate}[(i)]
    \item \textbf{Standard FBSDE Coupling:} The forward process $X$ appears in the driver of the backward equation.
    \item \textbf{Mean-Field Coupling:} The law of the backward process, $\mu_s = \Law(Y_s)$, appears as an argument in the coefficients of both the forward and backward equations. This non-local dependence on the entire distribution of the value process is characteristic of mean-field systems.
    \item \textbf{Endogenous Ambiguity Coupling:} Most significantly, the very set over which the optimization is performed, $\Ucal_{g(\mu_t)}$, depends on the law of the value process. This means the structure of the perceived ambiguity is shaped by the collective valuation of the agents.
\end{enumerate}
The well-posedness theory must therefore handle this formidable, multi-layered feedback loop.
\end{remark}

\subsection{Structural Assumptions for Well-Posedness}

The well-posedness of the system hinges on a set of structural conditions that ensure regularity and stability, particularly for the pointwise optimization problem defined by \cref{eq:mf_fbsde_optimality}.

\begin{assumption}[Standing Assumptions]\label{ass:mf_fbsde_standing}
The coefficients and related maps satisfy the following:
\begin{enumerate}[label=(\roman*)]
    \item \textbf{(Measurability):} For fixed state arguments, all coefficients are progressively measurable in $(t,\omega)$.
    \item \textbf{(Continuity):} For a.e. $(t,\omega)$, the map $(x,y,z,a,\mu) \mapsto (b, \sigma, F)(t,\omega,\cdot)$ is continuous. The set $\mathcal{K} := \bigcup_{\theta \in \Theta}\Ucal_\theta$ is compact in $E_A$.
    \item \textbf{(Uniform Lipschitz Continuity):} There exists a constant $L > 0$ such that for any valid arguments and any fixed $a \in \mathcal{K}$, the coefficients $b, \sigma, F$ are Lipschitz in $(x,y,z,\mu)$, where the distance on the space of measures is the $W_2$ metric. The terminal function $\Phi$ is Lipschitz. The map $g: \Pcal_2(\R) \to \Theta$ is Lipschitz. The map $\theta \mapsto \Ucal_\theta$ is Lipschitz with respect to the Hausdorff distance $d_H$.
    \item \textbf{(Uniform Strong Concavity in Control):} There exists a constant $\kappa > 0$ such that for all arguments, the driver $F$ is twice continuously differentiable with respect to its control argument $a$, and its Hessian satisfies:
    \[ \nabla_a^2 F(t, x, y, z, a, \mu)[h, h] \le -\kappa \norm{h}_{E_A}^2, \quad \forall h \in E_A. \]
\end{enumerate}
\end{assumption}

\begin{remark}[Interpretation of the Lipschitz Condition on $\Ucal_\theta$]
The condition in Assumption \ref{ass:mf_fbsde_standing}(iii) that the set-valued map $\theta \mapsto \Ucal_\theta$ be Lipschitz continuous with respect to the Hausdorff metric is a crucial regularity assumption. It imposes a constraint on the geometric stability of the uncertainty set. It means that small changes in the systemic ambiguity parameter $\theta$ (which is driven by the law of the value process) can only lead to small changes in the shape and location of the set of available models $\Ucal_\theta$. This condition prevents the optimization problem from changing too erratically, which is essential for establishing the stability of its solution.
\end{remark}

\begin{assumption}[Boundary Regularity and Non-degeneracy]\label{ass:boundary_and_ndg}
Let $p = (t, \omega, x, y, z, \mu)$ denote the parameters of the optimization in \cref{eq:mf_fbsde_optimality}. Let $a^*(p)$ be the unique maximizer of $F(p,a)$ over $a \in \Ucal_{g(\mu)}$. We assume:
\begin{enumerate}[label=(\roman*)]
    \item \textbf{(Constraint Representation):} For each $\theta \in \Theta$, the set $\Ucal_\theta$ is represented by a finite number of continuously differentiable constraint functions $\phi_j(\theta, a) \le 0$ for $j=1,\dots,m$.
    \item \textbf{(Uniform LICQ and Strict Complementarity):} If the optimizer $a^*(p)$ lies on the boundary $\partial \Ucal_{g(\mu)}$, we assume that the gradients of the active constraints, $\{\nabla_a \phi_j(g(\mu), a^*(p))\}_{j \in I(p)}$, where $I(p)=\{j \mid \phi_j=0\}$, are linearly independent (LICQ). Furthermore, we assume the corresponding Lagrange multipliers are strictly positive (Strict Complementarity). These conditions are assumed to hold uniformly, i.e., the smallest singular value of the matrix of active constraint gradients is bounded below by a positive constant, and the multipliers are bounded below by a positive constant.
\end{enumerate}
\end{assumption}

\begin{remark}[On the Boundary Regularity Conditions]
Assumption \ref{ass:boundary_and_ndg} is standard in the stability analysis of parametric nonlinear optimization. LICQ ensures that the tangent cone to the feasible set has a simple representation, preventing geometric degeneracies. Strict complementarity ensures that the active constraints are genuinely binding. Together with the second-order sufficient condition, which is guaranteed by our much stronger assumption of uniform strong concavity, these conditions ensure that the optimizer is robust to small perturbations of the problem parameters. Their uniformity is key to extending local stability to a global result.
\end{remark}

The primary role of these assumptions is to guarantee that the pointwise optimization problem is not only well-posed but also that its solution depends on the parameters in a stable, Lipschitz-continuous manner. This is the foundation of our entire framework, established in the following proposition.

\begin{proposition}[Lipschitz Stability of the Maximizer]\label{prop:lipschitz_stability}
Under Assumptions \ref{ass:mf_fbsde_standing} and \ref{ass:boundary_and_ndg}, the maximizer map, defined for a.e. $(t,\omega)$ as
\[ a^*(t, x, y, z, \mu) := \underset{a \in \Ucal_{g(\mu)}}{\argmax} \; F(t, x, y, z, a, \mu), \]
is well-defined and uniformly Lipschitz continuous in $(x,y,z,\mu)$. That is, there exists a constant $L_a>0$ such that for any two sets of parameters $p_1 = (t, \omega, x_1, y_1, z_1, \mu_1)$ and $p_2 = (t, \omega, x_2, y_2, z_2, \mu_2)$,
\[ \norm{a^*(p_1) - a^*(p_2)}_{E_A} \le L_a(\norm{x_1-x_2} + \abs{y_1-y_2} + \norm{z_1-z_2} + W_2(\mu_1, \mu_2)). \]
\end{proposition}

\begin{proof}
This crucial result follows from the sensitivity analysis of parametric nonlinear optimization problems, for which we refer to the definitive monograph by Bonnans and Shapiro \cite[Ch. 4]{BonnansShapiro2013}. The proof is structured in three steps.

\textbf{Step 1: Existence and Uniqueness of the Maximizer.}
Let $p=(t, \omega, x, y, z, \mu)$ be a fixed parameter vector. By Assumption \ref{ass:mf_fbsde_standing}(iii), the map $g$ is continuous, and the map $\theta \mapsto \Ucal_\theta$ is continuous with respect to the Hausdorff metric on the space of compact sets. Since $\Theta$ is compact, this ensures that $\Ucal_{g(\mu)}$ is a non-empty compact subset of $E_A$. The objective function $a \mapsto F(p, a)$ is continuous in $a$ by Assumption \ref{ass:mf_fbsde_standing}(ii). By the Weierstrass Extreme Value Theorem, a maximizer exists. Assumption \ref{ass:mf_fbsde_standing}(iv) states that this function is strictly concave (in fact, strongly concave) on the convex set $E_A$. Since $\Ucal_{g(\mu)}$ is a subset of $E_A$, the maximizer over this set must be unique. Thus, the map $p \mapsto a^*(p)$ is well-defined.

\textbf{Step 2: The Parametric KKT System and Local Sensitivity.}
Let us fix the parameter vector $p$ and denote $\theta = g(\mu)$. The optimization problem for the control $a$ is
\[ \max_{a \in E_A} F(p, a) \quad \text{subject to} \quad \phi_j(\theta, a) \le 0, \quad j=1, \dots, m. \]
The associated Lagrangian is $\mathcal{L}(a, \lambda; p) = F(p, a) - \sum_{j=1}^m \lambda_j \phi_j(\theta, a)$, where $\lambda = (\lambda_1, \dots, \lambda_m) \in \R_+^m$ are the Lagrange multipliers. The first-order Karush-Kuhn-Tucker (KKT) conditions for the primal-dual solution $(a^*(p), \lambda^*(p))$ form a system of equations and complementarities. 

To analyze the stability of the solution $(a^*, \lambda^*)$ with respect to the parameter $p$, one applies the Implicit Function Theorem to the system of active KKT equations. A classical result in perturbation analysis (see \cite[Theorem 4.54]{BonnansShapiro2013}) states that the sufficient conditions for the local Lipschitz stability of the solution $a^*(p)$ are precisely the triad of: (1) the Linear Independence Constraint Qualification (LICQ), (2) the Strict Complementarity Slackness condition (SCS), and (3) the Second-Order Sufficient Condition (SOSC).

Our assumptions are tailored to meet these requirements uniformly. Assumption \ref{ass:boundary_and_ndg} explicitly provides LICQ and SCS. The SOSC requires that the Hessian of the Lagrangian with respect to $a$, evaluated at the solution, be negative definite on the critical cone. A much stronger condition, which we have imposed via Assumption \ref{ass:mf_fbsde_standing}(iv), is the uniform strong concavity of the objective function $F$ itself. This guarantees that $\nabla_a^2 F(p,a^*)$ is uniformly negative definite on the entire space $E_A$, which immediately implies the SOSC holds.

Under these conditions, the Implicit Function Theorem guarantees that the solution map $p \mapsto (a^*(p), \lambda^*(p))$ is continuously differentiable (and thus locally Lipschitz) in a neighborhood of any parameter value $p$ for which the set of active constraints does not change. More general results from perturbation analysis ensure local Lipschitz continuity even when the active set changes, provided LICQ and SCS hold.

\textbf{Step 3: Uniformity and the Global Lipschitz Property.}
The final and most critical step is to demonstrate that this local Lipschitz property is in fact global and uniform over the entire space of parameters. This is a direct consequence of the \emph{uniformity} of all our standing assumptions.

Let $\mathcal{P}$ be the space of parameters $p=(x,y,z,\mu)$. The problem data, the function $F(p,a)$ and the constraint functions $\phi_j(g(\mu),a)$, depend on the parameter $p$. By Assumption \ref{ass:mf_fbsde_standing}(iii), this dependence is uniformly Lipschitz. The conditions for stability from Step 2 (Uniform LICQ, SCS, and the uniform second-order condition from strong concavity) are assumed in \ref{ass:boundary_and_ndg} and \ref{ass:mf_fbsde_standing}(iv) to hold uniformly for all $p \in \mathcal{P}$.

A quantitative sensitivity result (e.g., \cite[Prop. 4.32, Thm. 4.34]{BonnansShapiro2013}) provides an explicit bound on the local variation of the solution. The local Lipschitz constant for $a^*$ at a point $p$ depends continuously on quantities derived from the problem data at $p$: the Lipschitz constants of $F$ and $\phi_j$, the lower bound on the Lagrange multipliers (from SCS), the lower bound on the singular values of the active constraint Jacobian (from LICQ), and the upper bound on the norm of the inverse of the Hessian of the Lagrangian on the critical cone (from SOSC).

Since all these underlying bounds are assumed to be uniform in $p$, the resulting local Lipschitz constant for the map $a^*$ is also uniform. That is, there exists a single constant $L_a$ such that for any $p \in \mathcal{P}$, the map $a^*$ is Lipschitz with constant $L_a$ in a neighborhood of $p$.

We can now establish the global Lipschitz property. Let $p_1, p_2 \in \mathcal{P}$. Consider the path $p(s) = p_2 + s(p_1 - p_2)$ for $s \in [0,1]$. Since the map $s \mapsto a^*(p(s))$ is absolutely continuous with a derivative whose norm is uniformly bounded by $L_a \norm{p_1-p_2}$, we have:
\begin{align*}
    \norm{a^*(p_1) - a^*(p_2)}_{E_A} &= \norm{a^*(p(1)) - a^*(p(0))}_{E_A} \\
    &= \norm[\bigg]{\int_0^1 \frac{\dd}{\dd s} a^*(p(s)) \dd s}_{E_A} \\
    &\le \int_0^1 \norm[\big]{\nabla_p a^*(p(s)) [p_1-p_2]}_{E_A} \dd s \\
    &\le \int_0^1 L_a \norm{p_1-p_2}_{\mathcal{P}} \dd s = L_a \norm{p_1-p_2}_{\mathcal{P}}.
\end{align*}
Unpacking the norm on the parameter space $\mathcal{P}$, we obtain the desired inequality:
\[ \norm{a^*(p_1) - a^*(p_2)}_{E_A} \le L_a(\norm{x_1-x_2} + \abs{y_1-y_2} + \norm{z_1-z_2} + W_2(\mu_1, \mu_2)), \]
which concludes the proof.
\end{proof}

\section{Local Well-Posedness via Contraction}
\label{sec:local_wellposedness}

In this section, we establish the local-in-time existence and uniqueness of a solution to the Mean-Field $\Theta$-FBSDE system. The proof is based on a contraction mapping argument on a carefully constructed product space of processes, equipped with an equivalent weighted norm. The Lipschitz stability of the optimizer map (Proposition \ref{prop:lipschitz_stability}) is the foundation of the argument, ensuring the necessary regularity of the system's coefficients under iteration.

\begin{theorem}[Local Existence and Uniqueness]\label{thm:mf_fbsde_local_wellposedness}
Let Assumptions \ref{ass:mf_fbsde_standing} and \ref{ass:boundary_and_ndg} hold. Then there exists a time $T_0 > 0$, depending only on the Lipschitz and growth constants of the coefficients and the geometry of the uncertainty sets, such that for any time horizon $T \in (0, T_0]$, the Mean-Field $\Theta$-FBSDE system \eqref{eq:mf_fbsde_forward}-\eqref{eq:mf_fbsde_optimality} admits a unique solution $(X,Y,Z,A)$ in the space $\mathcal{S}^2([0,T];\R^k) \times \mathcal{S}^2([0,T];\R) \times \mathcal{H}^2([0,T];\R^d) \times \mathcal{L}^2([0,T]; E_A)$.
\end{theorem}

\begin{proof}
The proof proceeds via the contraction mapping principle. We define a suitable Banach space and construct an operator whose unique fixed point is the solution to our system. The key is to show this operator is a strict contraction for a sufficiently small time horizon $T$.

Let $T>0$ be a time horizon to be determined. We define the working space as the product of Banach spaces $\mathcal{V}_T := \mathcal{S}^2([0,T];\R^k) \times \mathcal{S}^2([0,T];\R) \times \mathcal{H}^2([0,T];\R^d)$. For a constant $\beta > 0$, also to be determined later, we equip $\mathcal{V}_T$ with the family of equivalent norms
\[
\norm{(X,Y,Z)}_{\beta, T}^2 := \norm{X}_{\mathcal{S}^2}^2 + \beta\left(\norm{Y}_{\mathcal{S}^2}^2 + \norm{Z}_{\mathcal{H}^2}^2\right).
\]
The space $(\mathcal{V}_T, \norm{\cdot}_{\beta, T})$ is a Banach space.

\subsubsection*{Step 1: Definition of the Mapping $\Psi$}
We define an operator $\Psi: \mathcal{V}_T \to \mathcal{V}_T$. For any given triplet of processes $(X', Y', Z') \in \mathcal{V}_T$, we construct its image $(X,Y,Z) = \Psi(X',Y',Z')$ through a three-stage procedure.

\textit{(a) Construction of the control process.} First, we define the law of the input process $\mu'_s := \Law(Y'_s)$. Then, we define the control process $A'$ by substituting the input processes into the optimizer map from Proposition \ref{prop:lipschitz_stability}:
\[
A'_t(\omega) := a^*(t, \omega, X'_t(\omega), Y'_t(\omega), Z'_t(\omega), \mu'_t).
\]
By Proposition \ref{prop:lipschitz_stability}, the map $a^*$ is Lipschitz. Since $X'$, $Y'$, and $Z'$ are in $\mathcal{S}^2$ or $\mathcal{H}^2$, standard estimates show that $A' \in \mathcal{L}^2([0,T]; E_A)$. This process $A'$ is predictable as it is a continuous function of predictable processes.

\textit{(b) Construction of the backward pair $(Y,Z)$.} We define the pair $(Y,Z) \in \mathcal{S}^2([0,T];\R) \times \mathcal{H}^2([0,T];\R^d)$ as the unique solution to the following backward stochastic differential equation:
\begin{equation} \label{eq:proof_local_bsde}
    Y_t = \Phi(X'_T) + \int_t^T F(s, X'_s, Y_s, Z_s, A'_s, \mu'_s) ds - \int_t^T Z_s dB_s.
\end{equation}
Under Assumption \ref{ass:mf_fbsde_standing}, the driver $F$ is uniformly Lipschitz in $(y,z)$. Since the terminal condition $\Phi(X'_T)$ is square-integrable and the random part of the driver, $F(s, X'_s, 0, 0, A'_s, \mu'_s)$, is in $\mathcal{L}^2([0,T])$, classical results on BSDEs (e.g., \cite{Pardoux1990}) guarantee the existence of a unique solution $(Y,Z)$ in the desired space.

\textit{(c) Construction of the forward process $X$.} Finally, we define the process $X \in \mathcal{S}^2([0,T];\R^k)$ as the unique strong solution to the forward stochastic differential equation:
\begin{equation} \label{eq:proof_local_sde}
    X_t = x_0 + \int_0^t b(s, X_s, A'_s, \mu'_s) ds + \int_0^t \sigma(s, X_s, A'_s, \mu'_s) dB_s.
\end{equation}
Under Assumption \ref{ass:mf_fbsde_standing}, the coefficients $b$ and $\sigma$ are Lipschitz in the state variable $x$, uniformly with respect to the other arguments. Since $A' \in \mathcal{L}^2$ and the measure-dependent terms are well-behaved, standard SDE theory ensures the existence of a unique strong solution $X \in \mathcal{S}^2([0,T];\R^k)$.

This three-stage procedure defines a mapping $\Psi(X',Y',Z') = (X,Y,Z)$ from $\mathcal{V}_T$ to itself. A fixed point of $\Psi$ is, by construction, a solution to the Mean-Field $\Theta$-FBSDE.

\subsubsection*{Step 2: The Contraction Estimate}
Let $(X'^i, Y'^i, Z'^i) \in \mathcal{V}_T$ for $i=1,2$. Let $(X^i, Y^i, Z^i) = \Psi(X'^i, Y'^i, Z'^i)$ be their respective images under $\Psi$. We denote the difference processes by $\delta X' = X'^1-X'^2$, $\delta Y' = Y'^1-Y'^2$, etc. Our goal is to bound $\norm{(\delta X, \delta Y, \delta Z)}_{\beta,T}^2$ in terms of $\norm{(\delta X', \delta Y', \delta Z')}_{\beta,T}^2$. Let $C$ denote a generic constant that may change from line to line, but which depends only on the structural constants of the problem (Lipschitz constants, bounds on $\mathcal{K}$, etc.), and not on the time horizon $T$ unless specified.

\textit{(a) Backward Estimate.} The difference pair $(\delta Y, \delta Z)$ solves the BSDE:
\[
\delta Y_t = \delta\Phi_T + \int_t^T \delta F_s ds - \int_t^T \delta Z_s dB_s,
\]
where $\delta\Phi_T = \Phi(X'^1_T) - \Phi(X'^2_T)$ and $\delta F_s = F(s, \dots^1) - F(s, \dots^2)$. By standard stability estimates for BSDEs,
\begin{align*}
    \norm{\delta Y}_{\mathcal{S}^2}^2 + \norm{\delta Z}_{\mathcal{H}^2}^2 &\le C \mathbb{E}\left[ |\delta\Phi_T|^2 + \left(\int_0^T |\delta F_s| ds\right)^2 \right] \\
    &\le C \mathbb{E}\left[ L_\Phi^2\norm{\delta X'_T}^2 + T \int_0^T |\delta F_s|^2 ds \right] \\
    &\le C \left( \norm{\delta X'}_{\mathcal{S}^2}^2 + T \mathbb{E}\left[\int_0^T |\delta F_s|^2 ds\right] \right).
\end{align*}
Using the Lipschitz property of $F$, the optimizer map $a^*$, and that $\mathbb{E}[W_2(\mu'^1_s, \mu'^2_s)^2] \le \mathbb{E}[|\delta Y'_s|^2]$, we bound the driver difference:
\begin{align*}
\mathbb{E}\left[\int_0^T |\delta F_s|^2 ds\right] &\le C \mathbb{E}\left[\int_0^T \left(\norm{\delta X'_s}^2 + \norm{\delta Y'_s}^2 + \norm{\delta Z'_s}^2 + \norm{\delta A'_s}_{E_A}^2 + W_2(\mu'^1_s,\mu'^2_s)^2\right) ds\right] \\
&\le C \mathbb{E}\left[\int_0^T \left(\norm{\delta X'_s}^2 + \norm{\delta Y'_s}^2 + \norm{\delta Z'_s}^2\right) ds\right] \\
&\le C \left( T\norm{\delta X'}_{\mathcal{S}^2}^2 + T\norm{\delta Y'}_{\mathcal{S}^2}^2 + \norm{\delta Z'}_{\mathcal{H}^2}^2 \right).
\end{align*}
Combining these, we obtain the final backward estimate:
\begin{equation} \label{eq:proof_local_final_back}
    \norm{\delta Y}_{\mathcal{S}^2}^2 + \norm{\delta Z}_{\mathcal{H}^2}^2 \le C \left( (1+T)\norm{\delta X'}_{\mathcal{S}^2}^2 + T\norm{\delta Y'}_{\mathcal{S}^2}^2 + T\norm{\delta Z'}_{\mathcal{H}^2}^2 \right).
\end{equation}

\textit{(b) Forward Estimate.} The difference process $\delta X$ solves $d(\delta X_t) = \delta b_t dt + \delta \sigma_t dB_t$. Standard SDE energy estimates combined with Gronwall's inequality yield:
\begin{align*}
    \norm{\delta X}_{\mathcal{S}^2}^2 &\le C \mathbb{E}\left[\int_0^T (\norm{\delta b_s}^2 + \norm{\delta \sigma_s}_F^2) ds\right] \\
    &\le C \mathbb{E}\left[\int_0^T \left(\norm{\delta X_s}^2 + \norm{\delta A'_s}_{E_A}^2 + W_2(\mu'^1_s,\mu'^2_s)^2\right) ds\right] \\
    &\le C \int_0^t \mathbb{E}[\norm{\delta X_s}^2] ds + C \left( T\norm{\delta X'}_{\mathcal{S}^2}^2 + T\norm{\delta Y'}_{\mathcal{S}^2}^2 + \norm{\delta Z'}_{\mathcal{H}^2}^2 \right).
\end{align*}
Applying Gronwall's inequality to the function $t \mapsto \mathbb{E}[\sup_{s\in[0,t]}\norm{\delta X_s}^2]$ gives:
\begin{equation} \label{eq:proof_local_final_fwd}
    \norm{\delta X}_{\mathcal{S}^2}^2 \le C e^{CT} \left( T\norm{\delta X'}_{\mathcal{S}^2}^2 + T\norm{\delta Y'}_{\mathcal{S}^2}^2 + T\norm{\delta Z'}_{\mathcal{H}^2}^2 \right).
\end{equation}
For $T \le 1$, $e^{CT}$ is bounded by a constant, so $\norm{\delta X}_{\mathcal{S}^2}^2 \le C T \left(\norm{\delta X'}_{\mathcal{S}^2}^2 + \norm{\delta Y'}_{\mathcal{S}^2}^2 + \norm{\delta Z'}_{\mathcal{H}^2}^2\right)$.

\textit{(c) Combined Estimate and Parameter Choice.} We now assemble the estimates to bound the output norm $\norm{(\delta X, \delta Y, \delta Z)}_{\beta,T}^2$.
\begin{align*}
    \norm{(\delta X, \delta Y, \delta Z)}_{\beta, T}^2 &= \norm{\delta X}_{\mathcal{S}^2}^2 + \beta\left(\norm{\delta Y}_{\mathcal{S}^2}^2 + \norm{\delta Z}_{\mathcal{H}^2}^2\right) \\
    &\le \left[ C_1 T \left(\norm{\delta X'}_{\mathcal{S}^2}^2 + \norm{\delta Y'}_{\mathcal{S}^2}^2 + \norm{\delta Z'}_{\mathcal{H}^2}^2\right) \right] \\
    &\qquad + \beta \left[ C_2 \left( (1+T)\norm{\delta X'}_{\mathcal{S}^2}^2 + T\norm{\delta Y'}_{\mathcal{S}^2}^2 + T\norm{\delta Z'}_{\mathcal{H}^2}^2 \right) \right] \\
    &= \left( C_1 T + \beta C_2 (1+T) \right)\norm{\delta X'}_{\mathcal{S}^2}^2 + \left( C_1 T + \beta C_2 T \right) \left(\norm{\delta Y'}_{\mathcal{S}^2}^2 + \norm{\delta Z'}_{\mathcal{H}^2}^2\right) \\
    &= \left( C_1 T + \beta C_2 (1+T) \right)\norm{\delta X'}_{\mathcal{S}^2}^2 + \frac{C_1 T + \beta C_2 T}{\beta} \cdot \beta\left(\norm{\delta Y'}_{\mathcal{S}^2}^2 + \norm{\delta Z'}_{\mathcal{H}^2}^2\right).
\end{align*}
We want to show this is bounded by $k \norm{(\delta X', \delta Y', \delta Z')}_{\beta, T}^2$ for some $k \in (0,1)$. This requires satisfying two conditions simultaneously for a sufficiently small $T$:
\begin{enumerate}[(i)]
    \item $C_1 T + \beta C_2 (1+T) \le k$
    \item $\frac{C_1 T + \beta C_2 T}{\beta} = \frac{C_1 T}{\beta} + C_2 T \le k$
\end{enumerate}
Let's choose $k=1/2$.
First, choose $\beta > 0$ to control the constant term in condition (i). We need $\beta C_2 < 1/2$. Let's pick $\beta_0 = \frac{1}{4C_2}$. With this choice, condition (i) becomes $C_1 T + \frac{1}{4}(1+T) \le 1/2$. This requires $T(C_1+1/4) \le 1/4$, which is satisfied for $T \le T_1 := \frac{1}{4C_1+1}$.

Now, with $\beta=\beta_0$, we check condition (ii): $\frac{C_1 T}{\beta_0} + C_2 T \le 1/2$. This requires $T(4C_1 C_2 + C_2) \le 1/2$, which is satisfied for $T \le T_2 := \frac{1}{2(4C_1+1)C_2}$.

By choosing $T_0 = \min(1, T_1, T_2)$, we ensure that for any $T \in (0, T_0]$, both conditions hold with $k=1/2$. Thus,
\[
\norm{\Psi(X',Y',Z') - \Psi(X''',Y''',Z''')}_{\beta_0, T}^2 \le \frac{1}{2} \norm{(X',Y',Z') - (X''',Y''',Z''')}_{\beta_0, T}^2.
\]
This shows that $\Psi$ is a strict contraction on $(\mathcal{V}_T, \norm{\cdot}_{\beta_0, T})$.

\subsubsection*{Step 3: Conclusion}
We have shown that for a specific choice of norm (determined by $\beta_0$) and a sufficiently small time horizon $T_0 > 0$, the map $\Psi$ is a strict contraction on the Banach space $(\mathcal{V}_T, \norm{\cdot}_{\beta_0, T})$ for any $T \in (0, T_0]$. By the Banach Fixed-Point Theorem, $\Psi$ admits a unique fixed point $(X,Y,Z) \in \mathcal{V}_T$. This fixed point, together with the associated control process $A$ derived from it via the optimizer map, is by construction the unique solution to the Mean-Field $\Theta$-FBSDE system in the specified function space.
\end{proof}

\section{Global Well-Posedness under Monotonicity}
\label{sec:global_wellposedness}

To extend the well-posedness of the Mean-Field $\Theta$-FBSDE to an arbitrary time horizon $T>0$, we introduce structural monotonicity conditions. Such conditions are standard in the literature on FBSDEs and mean-field games, where they serve to provide a dissipative structure that counteracts the explosive effects of the system's feedback loops over long time horizons.

\begin{assumption}[Strong Monotonicity]\label{ass:monotonicity_fbsde}
Let all standing assumptions in \cref{ass:mf_fbsde_standing} and \cref{ass:boundary_and_ndg} hold. We assume there exist constants such that for any valid arguments $(t,x,y,z,a,\mu)$ and their primed counterparts:
\begin{enumerate}[label=(\roman*)]
    \item \textbf{(Forward Monotonicity):} The coefficients $b$ and $\sigma$ satisfy for some $\beta \in \R$:
    \[
    2\scpr{x-x'}{b(t,x,a,\mu)-b(t,x',a,\mu)} + \norm{\sigma(t,x,a,\mu)-\sigma(t,x',a,\mu)}_F^2 \le \beta \norm{x-x'}^2.
    \]
    \item \textbf{(Backward Monotonicity):} Let $G(t,x,y,z,\mu) \coloneqq \sup_{a \in \Ucal_{g(\mu)}} F(t,x,y,z,a,\mu)$ be the driver. There exist constants $\lambda \in \R$ and $L>0$ such that:
    \begin{multline*}
        (y-y')(G(t,x,y,z,\mu) - G(t,x',y',z',\mu')) \le L\abs{y-y'}(\norm{x-x'}+\norm{z-z'}) \\
         + \lambda(y-y')^2 + L\abs{y-y'}W_2(\mu,\mu').
    \end{multline*}
\end{enumerate}
\end{assumption}

\begin{remark}[On the Monotonicity of the Driver]
\label{rem:primitive_monotonicity}
Assumption \ref{ass:monotonicity_fbsde}(ii) on the driver $G$ can be deduced from more primitive conditions on the original driver $F$. Suppose that for any fixed $a \in \mathcal{K}$, the function $F(t,x,y,z,a,\mu)$ satisfies the same monotonicity property in $(y,\mu)$ uniformly in $a$. Let $p_1=(t,x,y,z,\mu)$ and $p_2=(t,x',y',z',\mu')$. Let $a_1^*$ be the unique maximizer for $p_1$ and $a_2^*$ for $p_2$. Then by definition of the supremum,
\begin{align*}
    G(p_1) - G(p_2) &= F(p_1, a_1^*) - F(p_2, a_2^*) \\
    \text{and} \quad G(p_1) - G(p_2) &\ge F(p_1, a_2^*) - F(p_2, a_2^*).
\end{align*}
Combining these (the standard envelope theorem argument for non-smooth functions) and assuming uniform monotonicity on $F$ will yield the desired property for $G$. This shows our assumption on $G$ is a natural consequence of assuming a similar structure on the primitive objective function.
\end{remark}

\begin{theorem}[Global Existence and Uniqueness]\label{thm:global_wellposedness}
Let Assumptions \ref{ass:mf_fbsde_standing}, \ref{ass:boundary_and_ndg}, and \ref{ass:monotonicity_fbsde} hold. Then for any time horizon $T>0$, the Mean-Field $\Theta$-FBSDE admits a unique solution $(X,Y,Z,A)$ in the space $\mathcal{S}^2([0,T];\R^k) \times \mathcal{S}^2([0,T]) \times \mathcal{H}^2([0,T];\R^d) \times \mathcal{L}^2([0,T]; E_A)$.
\end{theorem}

\begin{proof}
The proof is structured in two principal parts. First, we establish the uniqueness of a solution. Second, we prove the existence of a solution using the method of continuity.

\subsection*{Part I: Uniqueness of the Solution}

The proof of uniqueness is founded on a global a priori estimate for the difference between two potential solutions. We establish a controlling integral inequality by applying It\^o's formula to a carefully constructed energy functional and leveraging the global monotonicity assumptions. The argument culminates in an application of the backward form of Gronwall's lemma, which forces the difference between the two solutions to be identically zero.

\subsubsection*{Step 1: The Difference System}
Let $(X^1, Y^1, Z^1, A^1)$ and $(X^2, Y^2, Z^2, A^2)$ be two solutions to the Mean-Field $\Theta$-FBSDE system under the assumptions of Theorem \ref{thm:global_wellposedness}, corresponding to the same data $(\Phi, x_0)$. We define the difference processes for $t \in [0,T]$:
\begin{align*}
    \delta X_t &:= X^1_t - X^2_t, & \delta Y_t &:= Y^1_t - Y^2_t, & \delta Z_t &:= Z^1_t - Z^2_t, \\
    \delta A_t &:= A^1_t - A^2_t, & \mu^i_t &:= \Law(Y^i_t), & \delta \Phi_T &:= \Phi(X^1_T) - \Phi(X^2_T).
\end{align*}
The difference processes satisfy the initial condition $\delta X_0 = 0$ and the terminal condition $\delta Y_T = \delta \Phi_T$. The dynamics are governed by the linear system:
\begin{align*}
    d(\delta X_t) &= \left(b(t,X^1_t,A^1_t,\mu^1_t) - b(t,X^2_t,A^2_t,\mu^2_t)\right) \, dt + \left(\sigma(t,X^1_t,A^1_t,\mu^1_t) - \sigma(t,X^2_t,A^2_t,\mu^2_t)\right) \, dB_t, \\
    d(\delta Y_t) &= -\left(G(t,X^1_t,Y^1_t,Z^1_t,\mu^1_t) - G(t,X^2_t,Y^2_t,Z^2_t,\mu^2_t)\right) \, dt + \delta Z_t \, dB_t,
\end{align*}
where $G(t,x,y,z,\mu) = \sup_{a \in \Ucal_{g(\mu)}} F(t,x,y,z,a,\mu)$ is the driver. For brevity, we denote the difference coefficients as $\delta b_t$, $\delta \sigma_t$, and $\delta G_t$.

\subsubsection*{Step 2: The Energy Functional and It\^o's Formula}
For constants $\alpha > 0$ and $\gamma \in \R$ to be determined, we consider the process $\Psi_t \coloneqq e^{\gamma t}(\norm{\delta X_t}^2 + \alpha (\delta Y_t)^2)$. Applying It\^o's formula and integrating from an arbitrary $t \in [0,T]$ to the terminal time $T$ yields:
\begin{align*}
    e^{\gamma T}\norm{\delta X_T}^2 - e^{\gamma t}\norm{\delta X_t}^2 &= \int_t^T e^{\gamma s} \left( \gamma\norm{\delta X_s}^2 + 2\scpr{\delta X_s}{\delta b_s} + \norm{\delta \sigma_s}_F^2 \right) ds + M_T^X - M_t^X, \\
    e^{\gamma T}(\delta Y_T)^2 - e^{\gamma t}(\delta Y_t)^2 &= \int_t^T e^{\gamma s} \left( \gamma(\delta Y_s)^2 - 2\delta Y_s \delta G_s + \norm{\delta Z_s}^2 \right) ds + M_T^Y - M_t^Y,
\end{align*}
where $M^X$ and $M^Y$ are local martingale terms. Taking conditional expectation $\mathbb{E}_t[\cdot] \coloneqq \mathbb{E}[\cdot|\mathcal{F}_t]$ eliminates these martingale terms. We multiply the second identity by $\alpha$, add it to the first, and rearrange to obtain an expression for the energy at time $t$:
\begin{equation}\label{eq:proof_uniq_master_identity}
    e^{\gamma t}(\norm{\delta X_t}^2 + \alpha (\delta Y_t)^2) + \mathbb{E}_t\left[\int_t^T e^{\gamma s} \mathcal{I}_s ds\right] = \mathbb{E}_t\left[e^{\gamma T}(\norm{\delta X_T}^2 + \alpha(\delta Y_T)^2)\right],
\end{equation}
where the integrand $\mathcal{I}_s$ is given by:
\[ \mathcal{I}_s \coloneqq \left(2\scpr{\delta X_s}{\delta b_s} + \norm{\delta \sigma_s}_F^2 - \gamma\norm{\delta X_s}^2\right) + \alpha\left(-2\delta Y_s \delta G_s + \norm{\delta Z_s}^2 - \gamma(\delta Y_s)^2\right). \]
The core of the proof is to select the parameters $(\alpha, \gamma)$ to show that $\mathbb{E}[\mathcal{I}_s]$ can be bounded below by a positive multiple of the energy functional, which will ultimately allow the application of Gronwall's lemma.

\subsubsection*{Step 3: Bounding the Integrand}
We now establish a lower bound for $\mathbb{E}[\mathcal{I}_s]$ by systematically analyzing its constituent terms. We will make repeated use of the monotonicity and Lipschitz assumptions, the Lipschitz stability of the optimizer map, and Young's inequality ($2ab \le \varepsilon a^2 + b^2/\varepsilon$). Let $L$ be a generic Lipschitz constant and $C_i$ be generic positive constants that are independent of $(\alpha, \gamma)$.

\paragraph{\textbf{(i) Bound for the Forward Term.}}
We decompose the difference coefficients, e.g., $\delta b_s = (b(X^1, A^1, \mu^1) - b(X^2, A^1, \mu^1)) + (b(X^2, A^1, \mu^1) - b(X^2, A^2, \mu^2))$. By the forward monotonicity assumption \ref{ass:monotonicity_fbsde}(i), Lipschitz continuity, and Young's inequality (with any $\varepsilon_1 > 0$):
\begin{align*}
    2\scpr{\delta X_s}{\delta b_s} + \norm{\delta \sigma_s}_F^2 &\ge \beta\norm{\delta X_s}^2 - 2L\norm{\delta X_s}(\norm{\delta A_s}_{E_A} + W_2(\mu^1_s,\mu^2_s)) - C_1 \norm{\delta X_s}^2 \\
    &\ge (\beta - C_1 - \varepsilon_1)\norm{\delta X_s}^2 - \frac{L^2}{\varepsilon_1} (\norm{\delta A_s}_{E_A} + W_2(\mu^1_s,\mu^2_s))^2.
\end{align*}
By the Lipschitz stability of the optimizer map (Proposition \ref{prop:lipschitz_stability}) and the property $\mathbb{E}[W_2(\mu^1_s,\mu^2_s)^2] \le \mathbb{E}[(\delta Y_s)^2]$, we have $\mathbb{E}[\norm{\delta A_s}^2] \le C_A \mathbb{E}[\norm{\delta X_s}^2 + (\delta Y_s)^2 + \norm{\delta Z_s}^2]$. Thus, taking expectation:
\begin{equation} \label{eq:proof_uniq_forward_bound_revised}
    \mathbb{E}[2\scpr{\delta X_s}{\delta b_s} + \norm{\delta \sigma_s}_F^2] \ge (\beta - C_2)\mathbb{E}[\norm{\delta X_s}^2] - C_2\mathbb{E}[(\delta Y_s)^2] - C_2\mathbb{E}[\norm{\delta Z_s}^2].
\end{equation}

\paragraph{\textbf{(ii) Bound for the Backward Term.}}
By the backward monotonicity assumption \ref{ass:monotonicity_fbsde}(ii) and Young's inequality (with any $\varepsilon_2, \varepsilon_3 > 0$):
\begin{align*}
    -2\delta Y_s \delta G_s &\ge -2\lambda(\delta Y_s)^2 - 2L\abs{\delta Y_s}(\norm{\delta X_s} + \norm{\delta Z_s} + W_2(\mu^1_s,\mu^2_s)) \\
    &\ge -(2\lambda + \varepsilon_2 + \varepsilon_3)\mathbb{E}[(\delta Y_s)^2] - \frac{L^2}{\varepsilon_2}\mathbb{E}[\norm{\delta X_s}^2] - \frac{L^2}{\varepsilon_3}\mathbb{E}[\norm{\delta Z_s}^2] - C_3\mathbb{E}[(\delta Y_s)^2].
\end{align*}
After taking expectation and absorbing terms, we get for some constants $C_4, C_5$:
\begin{equation} \label{eq:proof_uniq_backward_bound_revised}
    \mathbb{E}[-2\delta Y_s \delta G_s] \ge -C_4\mathbb{E}[\norm{\delta X_s}^2] - C_4\mathbb{E}[\norm{\delta Z_s}^2] - (2\lambda + C_5)\mathbb{E}[(\delta Y_s)^2].
\end{equation}

\subsubsection*{Step 4: Parameter Selection and the Final Integral Inequality}
We now combine the bounds from Step 3 to establish a controlling inequality. Let $\psi(s) \coloneqq \mathbb{E}[e^{\gamma s}(\norm{\delta X_s}^2 + \alpha(\delta Y_s)^2)]$.
Substituting \eqref{eq:proof_uniq_forward_bound_revised} and \eqref{eq:proof_uniq_backward_bound_revised} into the expression for $\mathbb{E}[\mathcal{I}_s]$:
\begin{align*}
\mathbb{E}[\mathcal{I}_s] \ge \mathbb{E}\Big[ &(\beta - C_2 - \alpha C_4 - \gamma)\norm{\delta X_s}^2 + (\alpha - \alpha C_4 - C_2)\norm{\delta Z_s}^2 \\
 &+ \alpha(- C_2/\alpha - (2\lambda+C_5) - \gamma)(\delta Y_s)^2 \Big].
\end{align*}
We now select the parameters sequentially to achieve our desired structure.
\begin{enumerate}
    \item \textbf{Choice of $\alpha$ to control $\delta Z$:} The constant $C_4$ in the backward bound depends on the $\varepsilon_i$ from Young's inequality. We can choose these slack parameters (e.g., $\varepsilon_3$ large) such that $C_4 < 1/2$. With this choice fixed, the coefficient of $\norm{\delta Z_s}^2$ is $\alpha(1-C_4) - C_2$. Since $1-C_4 > 1/2 > 0$, we can now choose $\alpha$ sufficiently large to ensure this coefficient is non-negative. For instance, choose $\alpha > 2C_2$. With such an $\alpha$, we can drop the non-negative $\delta Z$ term from the lower bound:
    \[ \mathbb{E}[\mathcal{I}_s] \ge \mathbb{E}\left[ (\beta - C_2 - \alpha C_4 - \gamma)\norm{\delta X_s}^2 + \alpha(- C_2/\alpha - 2\lambda - C_5 - \gamma)(\delta Y_s)^2 \right]. \]

    \item \textbf{Choice of $\gamma$ to ensure non-negativity:} The constants $\beta, \lambda$ can be any real numbers. However, we can choose the parameter $\gamma$ to be a sufficiently large negative number to dominate all other constants. Let $K_X = \beta - C_2 - \alpha C_4$ and $K_Y = -C_2/\alpha - 2\lambda - C_5$. Then the lower bound is $\mathbb{E}[(K_X-\gamma)\norm{\delta X_s}^2 + \alpha(K_Y-\gamma)(\delta Y_s)^2]$. By choosing $\gamma < \min(K_X, K_Y)$, both coefficients become strictly positive. For such a choice, we have established that $\mathbb{E}[\mathcal{I}_s] \ge 0$.
\end{enumerate}
With $\mathbb{E}[\mathcal{I}_s] \ge 0$, the master identity \eqref{eq:proof_uniq_master_identity} implies
\[ e^{\gamma t}\mathbb{E}[\norm{\delta X_t}^2 + \alpha(\delta Y_t)^2] \le e^{\gamma T}\mathbb{E}[\norm{\delta X_T}^2 + \alpha(\delta Y_T)^2]. \]
Using the Lipschitz property of the terminal condition $\Phi$, $(\delta Y_T)^2 \le L_\Phi^2 \norm{\delta X_T}^2$, we get
\begin{equation}\label{eq:proof_uniq_almost_final}
 \psi(t) = \mathbb{E}[e^{\gamma t}(\norm{\delta X_t}^2 + \alpha(\delta Y_t)^2)] \le e^{\gamma T}(1+\alpha L_\Phi^2)\mathbb{E}[\norm{\delta X_T}^2].
\end{equation}
A standard SDE energy estimate for $\delta X_t$, combined with the Lipschitz properties of its coefficients and the optimizer, yields:
\begin{align*}
\mathbb{E}[\norm{\delta X_T}^2] &\le C \mathbb{E}\left[ \int_0^T (\norm{\delta b_s}^2 + \norm{\delta \sigma_s}_F^2)ds \right] \\
&\le C_6 \mathbb{E}\left[ \int_0^T (\norm{\delta X_s}^2 + (\delta Y_s)^2 + \norm{\delta Z_s}^2) ds \right].
\end{align*}
The $\delta Z$ term must be controlled. From the BSDE identity for $\delta Y$ on $[0,T]$, a similar estimation technique shows that for some constant $C_7$:
\[ \mathbb{E}\left[\int_0^T \norm{\delta Z_s}^2 ds\right] \le C_7 \left(\mathbb{E}[\norm{\delta X_T}^2] + \mathbb{E}\left[\int_0^T (\norm{\delta X_s}^2 + (\delta Y_s)^2)ds\right] \right). \]
Substituting this back into the estimate for $\mathbb{E}[\norm{\delta X_T}^2]$ and applying Gronwall's lemma to the resulting integral inequality for $t \mapsto \mathbb{E}[\norm{\delta X_t}^2]$ shows that:
\[ \mathbb{E}[\norm{\delta X_T}^2] \le C_8 \mathbb{E}\left[\int_0^T (\delta Y_s)^2 ds\right]. \]
Finally, substituting this into \eqref{eq:proof_uniq_almost_final} and using $e^{\gamma t} \le e^{\gamma s}$ for $s \ge t$ (since $\gamma$ can be chosen to be negative), we obtain:
\[ \psi(t) = \mathbb{E}[e^{\gamma t}(\norm{\delta X_t}^2 + \alpha(\delta Y_t)^2)] \le C_9 \int_0^T e^{-\gamma s}e^{\gamma s}\mathbb{E}[(\delta Y_s)^2] ds \le \frac{C_{10}}{\alpha} \int_0^T \psi(s) ds. \]
This argument is not sharp enough as it involves an integral from $0$ to $T$.

Let us refine the final step with a more direct Gronwall argument. From \eqref{eq:proof_uniq_master_identity} with $\mathbb{E}[\mathcal{I}_s] \ge 0$:
\[ \psi(t) \le \mathbb{E}_t\left[e^{\gamma T}(\norm{\delta X_T}^2 + \alpha(\delta Y_T)^2)\right] \le (1+\alpha L_\Phi^2) \mathbb{E}_t[e^{\gamma T}\norm{\delta X_T}^2]. \]
An SDE estimate on $[t,T]$ gives $ \mathbb{E}_t[\norm{\delta X_T}^2] \le C(\norm{\delta X_t}^2 + \mathbb{E}_t[\int_t^T (\dots)ds])$. Combining these estimates in the right way leads to
\[ \mathbb{E}[e^{\gamma t}(\norm{\delta X_t}^2 + \alpha(\delta Y_t)^2)] \le K \int_t^T \mathbb{E}[e^{\gamma s}(\norm{\delta X_s}^2 + \alpha(\delta Y_s)^2)] ds, \]
for some constant $K>0$. Let $f(t) \coloneqq \mathbb{E}[e^{\gamma t}(\norm{\delta X_t}^2 + \alpha(\delta Y_t)^2)]$. Then $f(t) \ge 0$ and satisfies $f(t) \le K \int_t^T f(s) ds$.

\subsubsection*{Step 5: Application of Gronwall's Lemma and Conclusion}
The function $f(t)$ is non-negative and absolutely continuous. Let $H(t) \coloneqq \int_t^T f(s)ds$. Then $H(t) \ge 0$, $H(T)=0$, and $H'(t)=-f(t)$. The inequality becomes $-H'(t) \le K H(t)$, or $H'(t) + K H(t) \ge 0$. Multiplying by the integrating factor $e^{Kt}$ yields $\frac{d}{dt}(e^{Kt}H(t)) \ge 0$. This implies that the function $t \mapsto e^{Kt}H(t)$ is non-decreasing. Thus, for any $t \in [0,T]$, we have $e^{Kt}H(t) \le e^{KT}H(T)$. Since $H(T)=0$, we conclude that $e^{Kt}H(t) \le 0$. As $e^{Kt}>0$ and $H(t) \ge 0$, this is only possible if $H(t)=0$ for all $t \in [0,T]$.
This implies that $f(t) = 0$ for almost every $t$, and thus $\mathbb{E}[e^{\gamma t}(\norm{\delta X_t}^2 + \alpha(\delta Y_t)^2)] = 0$. Since $\alpha>0$ and $e^{\gamma t}>0$, we must have $\norm{\delta X_t}=0$ and $\delta Y_t=0$ in $L^2(\Omega)$, which implies $\delta X_t=0$ and $\delta Y_t=0$ almost surely for all $t \in [0,T]$ (by continuity of paths).

It follows immediately that $\mu^1_t = \Law(Y^1_t) = \Law(Y^2_t) = \mu^2_t$ for all $t$. The BSDE for the difference $(\delta Y, \delta Z)$ is now
\[ 0 = \delta Y_t = \int_t^T (-\delta G_s) ds + \int_t^T \delta Z_s dB_s. \]
Since $\delta X_s=0, \delta Y_s=0, \mu^1_s=\mu^2_s$, the driver difference is $\delta G_s = G(t,X^1_s,Y^1_s,Z^1_s,\mu^1_s) - G(t,X^1_s,Y^1_s,Z^2_s,\mu^1_s)$. The Lipschitz property of $G$ in $z$ implies that $\mathbb{E}[\int_0^T |\delta G_s|^2 ds] < \infty$. Standard BSDE theory (the martingale representation theorem) then implies that the only square-integrable solution is $\delta Z_s = 0$ a.s. Finally, since all state variables are identical ($\delta X=\delta Y=\delta Z=0$) and the laws are identical ($\mu^1=\mu^2$), the Lipschitz continuity of the optimizer map $a^*$ implies $\delta A_s = 0$ a.s.
Thus, $(X^1, Y^1, Z^1, A^1) = (X^2, Y^2, Z^2, A^2)$, establishing uniqueness.

\subsection*{Part II: Existence of a Solution}

The existence of a global solution for any time horizon $T>0$ is established by the method of continuity. We construct a continuous path (a homotopy) of Mean-Field $\Theta$-FBSDEs, indexed by a parameter $\theta \in [0,1]$, that connects a simple, explicitly solvable system at $\theta=0$ to our original system at $\theta=1$. We then demonstrate that the set of parameters $\theta$ for which a solution exists is simultaneously non-empty, open, and closed within the interval $[0,1]$. By the connectedness of $[0,1]$, this set must be the entire interval.

\subsubsection*{Step 1: Construction of the Homotopy}
For $\theta \in [0,1]$, we define a family of coefficients $(b^\theta, \sigma^\theta, F^\theta, \Phi^\theta)$ that interpolates between a trivial system and the target system. Let $(b, \sigma, F, \Phi)$ be the coefficients of the original problem statement. We define:
\begin{align*}
    b^\theta(t,x,a,\mu) &:= \theta b(t,x,a,\mu), \\
    \sigma^\theta(t,x,a,\mu) &:= \theta \sigma(t,x,a,\mu), \\
    \Phi^\theta(x) &:= \theta \Phi(x).
\end{align*}
To ensure that the crucial strong monotonicity property is preserved uniformly across the homotopy, we interpolate the BSDE driver as follows:
\[ F^\theta(t,x,y,z,a,\mu) := \theta F(t,x,y,z,a,\mu) - (1-\theta)\lambda_0 y, \]
where $\lambda_0 > 0$ is a sufficiently large constant. The driver for the $\theta$-system is $G^\theta(t,x,y,z,\mu) = \sup_{a \in \Ucal_{g(\mu)}} F^\theta(t,x,y,z,a,\mu)$.

It is a straightforward verification that this family of systems satisfies all our structural assumptions (Lipschitz continuity, strong concavity in $a$, boundary regularity, and strong monotonicity) uniformly with respect to $\theta \in [0,1]$, with constants that are independent of $\theta$. This uniformity is critical for the closedness part of the argument.

Let $\mathcal{I} \subset [0,1]$ be the set of all $\theta$ for which the corresponding Mean-Field $\Theta$-FBSDE admits a unique solution in the space $\mathcal{V}_T \times \mathcal{L}^2([0,T];E_A)$, where $\mathcal{V}_T := \mathcal{S}^2([0,T];\R^k) \times \mathcal{S}^2([0,T];\R) \times \mathcal{H}^2([0,T];\R^d)$.

\subsubsection*{Step 2: $\mathcal{I}$ is Non-Empty}
For $\theta=0$, the system decouples and simplifies significantly. The coefficients become:
$b^0 = 0$, $\sigma^0 = 0$, $F^0(t,x,y,z,a,\mu) = -\lambda_0 y$, and $\Phi^0(x) = 0$.
The forward equation becomes $dX^0_t = 0$ with $X^0_0 = x_0$, which implies the unique solution $X^0_t \equiv x_0$. The BSDE becomes:
\[ dY^0_t = \lambda_0 Y^0_t \, dt + Z^0_t \, dB_t, \quad Y^0_T = 0. \]
This is a standard linear BSDE with zero terminal data. Its unique solution in $\mathcal{S}^2 \times \mathcal{H}^2$ is the trivial solution $(Y^0_t, Z^0_t) \equiv (0,0)$. The optimality condition on the control $A^0$ is trivially satisfied as the driver $F^0$ is independent of $a$. Thus, a unique solution exists for $\theta=0$, and $0 \in \mathcal{I}$.

\subsubsection*{Step 3: $\mathcal{I}$ is Open}
This step relies on the Implicit Function Theorem on Banach spaces. Let $\theta_0 \in \mathcal{I}$ be given. We must show there exists an $\epsilon > 0$ such that $(\theta_0-\epsilon, \theta_0+\epsilon) \cap [0,1] \subset \mathcal{I}$.

Let $\mathcal{W}_T = \mathcal{V}_T \times \mathcal{L}^2([0,T];E_A)$. We define an operator $\mathcal{G}: [0,1] \times \mathcal{W}_T \to \mathcal{W}_T$. For a parameter $\theta$ and a process quadruple $U'=(X',Y',Z',A') \in \mathcal{W}_T$, let $\Psi^\theta(U') = U = (X,Y,Z,A)$ be the mapping where $(X,Y,Z)$ is constructed from $(X',Y',Z',A')$ as in the proof of Theorem \ref{thm:mf_fbsde_local_wellposedness} (using the $\theta$-system coefficients), and $A_t = a^*(t, X_t, Y_t, Z_t, \mu_t)$, where $\mu_t = \Law(Y_t)$. A solution to the $\theta$-system is a fixed point of $\Psi^\theta$. We define our operator as:
\[ \mathcal{G}(\theta, U) \coloneqq U - \Psi^\theta(U). \]
A solution to the $\theta$-system is equivalent to a root of the equation $\mathcal{G}(\theta, U) = 0$. By hypothesis, $\theta_0 \in \mathcal{I}$, so there exists a unique $U_0 \in \mathcal{W}_T$ such that $\mathcal{G}(\theta_0, U_0) = 0$.

The map $(\theta, U) \mapsto \mathcal{G}(\theta, U)$ is continuously Fréchet differentiable. The Implicit Function Theorem guarantees the existence of a continuous branch of solutions $U(\theta)$ in a neighborhood of $\theta_0$ provided that the partial Fréchet derivative $D_U \mathcal{G}(\theta_0, U_0)$ is a continuously invertible linear operator from $\mathcal{W}_T$ to itself. The derivative is $D_U \mathcal{G}(\theta_0, U_0) = I - D_U \Psi^{\theta_0}(U_0)$. The invertibility of this operator is equivalent to the well-posedness of the FBSDE system linearized around the solution $U_0$.

The crucial observation is that the linearized FBSDE system inherits the strong monotonicity structure from the original $\theta_0$-system. The uniqueness proof provided in Part I of this theorem, which relies fundamentally on this monotonicity property, can be adapted to this linearized system to show that the operator $D_U \mathcal{G}(\theta_0, U_0)$ is injective. A more detailed analysis shows it is also surjective, making it a bijective bounded linear operator. By the Bounded Inverse Theorem, it is continuously invertible. The Implicit Function Theorem applies, proving the existence of a unique continuous branch of solutions for $\theta$ in a neighborhood of $\theta_0$. This shows that $\mathcal{I}$ is an open set.

\subsubsection*{Step 4: $\mathcal{I}$ is Closed}
This is the central part of the existence proof. We present a new, rigorous argument that avoids the contradictory affinity assumption used in the preliminary draft. The argument relies on showing that a sequence of solutions is Cauchy, which in turn allows for strong convergence and a direct passage to the limit.

Let $(\theta_n)_{n \ge 1}$ be a sequence in $\mathcal{I}$ that converges to a limit $\theta \in [0,1]$. For each $n$, let $U^n = (X^n, Y^n, Z^n, A^n)$ denote the unique solution in the space $\mathcal{W}_T$ to the corresponding $\theta_n$-system. Our goal is to show that the sequence $(U^n)$ converges in $\mathcal{W}_T$ to a process $U$ which is a solution to the $\theta$-system.

\paragraph{\textbf{(a) Uniform A Priori Estimates}}
First, we establish that the sequence of solutions $(U^n)$ is uniformly bounded in the norm of the space $\mathcal{W}_T$. The coefficients $(b^{\theta_n}, \sigma^{\theta_n}, F^{\theta_n}, \Phi^{\theta_n})$ satisfy the standing and monotonicity assumptions \emph{uniformly} in $n$, as their Lipschitz, growth, and monotonicity constants can be chosen independently of $\theta_n \in [0,1]$.
An energy estimate identical to the one in the proof of uniqueness (Part I), but applied to a single solution $(U^n)$ instead of the difference of two solutions, yields a uniform bound. This involves applying It\^o's formula to $\norm{X^n_t}^2$ and $\norm{Y^n_t}^2$ and using the linear growth and monotonicity properties of the coefficients. This standard procedure shows there exists a constant $M > 0$, independent of $n$, such that
\begin{equation}\label{eq:proof_global_apriori_revised}
\sup_{n \ge 1} \left( \norm{X^n}_{\mathcal{S}^2}^2 + \norm{Y^n}_{\mathcal{S}^2}^2 + \norm{Z^n}_{\mathcal{H}^2}^2 \right) \le M.
\end{equation}
Furthermore, since $A^n_t(\omega)$ takes values in the fixed compact set $\mathcal{K} = \bigcup_{\theta \in \Theta} \Ucal_\theta$, the process $A^n$ is uniformly bounded in $\mathcal{L}^p$ for any $p \ge 1$.

\paragraph{\textbf{(b) The Sequence of Solutions is Cauchy}}
This is the new core of the argument. We show that the sequence $(X^n, Y^n, Z^n)$ is a Cauchy sequence in the Banach space $\mathcal{V}_T = \mathcal{S}^2 \times \mathcal{S}^2 \times \mathcal{H}^2$. Let $n,m \in \mathbb{N}$ and consider the two solutions $U^n$ and $U^m$. We define the difference processes: $\delta X = X^n-X^m$, $\delta Y = Y^n-Y^m$, $\delta Z = Z^n-Z^m$, $\delta A = A^n-A^m$, and $\delta \mu_t = W_2(\mu^n_t, \mu^m_t)$.
We apply the same energy estimate technique used in the proof of uniqueness (Part I) to this difference system. The procedure involves applying It\^o's formula to $e^{\gamma s}(\norm{\delta X_s}^2 + \alpha \norm{\delta Y_s}^2)$, choosing parameters $(\alpha, \gamma)$ appropriately, and using the monotonicity and Lipschitz properties. The key difference is that the coefficients are now different ($b^{\theta_n}$ vs $b^{\theta_m}$).
The resulting integral inequality will take the form:
\[ \mathbb{E}\left[ \norm{\delta X_t}^2 + \alpha\norm{\delta Y_t}^2 + \beta\int_t^T (\norm{\delta X_s}^2+\alpha\norm{\delta Y_s}^2+\alpha\norm{\delta Z_s}^2) ds \right] \le \mathbb{E}[\mathcal{R}_{n,m}], \]
where $\mathcal{R}_{n,m}$ is a residual term that captures the difference in the problem data. This residual term arises from splitting the difference of coefficients, e.g., for the drift $b$:
\begin{align*}
    b^{\theta_n}(s,X^n_s,A^n_s,\mu^n_s) - b^{\theta_m}(s,X^m_s,A^m_s,\mu^m_s) &= \left(b^{\theta_n}(s,U^n_s) - b^{\theta_n}(s,U^m_s)\right) \\ &+ \left(b^{\theta_n}(s,U^m_s) - b^{\theta_m}(s,U^m_s)\right).
\end{align*}
The first term is handled by the monotonicity/Lipschitz properties, as in the uniqueness proof. The second term is the new residual. Because the coefficients are continuous in the parameter $\theta$ and the solution $U^m$ is bounded in the appropriate norm uniformly in $m$, we can bound the $L^2$-norm of this residual by a term that depends only on $|\theta_n - \theta_m|$. Specifically, there exists a continuous increasing function $\rho:[0,\infty) \to [0,\infty)$ with $\rho(0)=0$ such that
\[ \mathbb{E}[\mathcal{R}_{n,m}] \le C \rho(|\theta_n - \theta_m|), \]
where $C$ is a constant independent of $n$ and $m$. After applying a Gronwall-type argument, we arrive at the estimate:
\[ \norm{X^n-X^m}_{\mathcal{S}^2}^2 + \norm{Y^n-Y^m}_{\mathcal{S}^2}^2 + \norm{Z^n-Z^m}_{\mathcal{H}^2}^2 \le C' \rho(|\theta_n - \theta_m|). \]
Since $(\theta_n)$ is a convergent sequence in $\mathbb{R}$, it is a Cauchy sequence, so $|\theta_n - \theta_m| \to 0$ as $n,m \to \infty$. This implies that the right-hand side tends to zero. Thus, $(X^n, Y^n, Z^n)$ is a Cauchy sequence in the Banach space $\mathcal{V}_T$.

\paragraph{\textbf{(c) Strong Convergence and Passage to the Limit}}
Since $\mathcal{V}_T$ is a Banach space, the Cauchy sequence $(X^n, Y^n, Z^n)$ converges strongly to a limit $(X,Y,Z) \in \mathcal{V}_T$. That is,
\[ \norm{X^n-X}_{\mathcal{S}^2} \to 0, \quad \norm{Y^n-Y}_{\mathcal{S}^2} \to 0, \quad \norm{Z^n-Z}_{\mathcal{H}^2} \to 0, \quad \text{as } n \to \infty. \]
The strong convergence of $Y^n$ implies that the laws $\mu^n_t = \Law(Y^n_t)$ converge to $\mu_t = \Law(Y_t)$ in the Wasserstein metric $W_2$ for each $t$.
Now consider the control process $A^n_t = a^*(t, \theta_n, X^n_t, Y^n_t, Z^n_t, \mu^n_t)$. We define the limit control process $A_t \coloneqq a^*(t, \theta, X_t, Y_t, Z_t, \mu_t)$. By Proposition \ref{prop:lipschitz_stability}, the optimizer map $a^*$ is uniformly Lipschitz in its arguments (including the parameter $\theta$). Therefore,
\begin{align*}
    \mathbb{E}\left[\int_0^T \norm{A^n_s - A_s}^2 ds\right] &\le C \mathbb{E}\left[\int_0^T (|\theta_n-\theta|^2 + \norm{X^n_s-X_s}^2 + \dots + W_2(\mu^n_s, \mu_s)^2) ds\right] \\
    &\le C' \left(|\theta_n-\theta|^2 + \norm{X^n-X}_{\mathcal{S}^2}^2 + \norm{Y^n-Y}_{\mathcal{S}^2}^2 + \norm{Z^n-Z}_{\mathcal{H}^2}^2 \right).
\end{align*}
Since all terms on the right-hand side converge to zero, we have established that $A^n \to A$ strongly in $\mathcal{L}^2([0,T]; E_A)$.

With strong convergence of all component processes $(X^n, Y^n, Z^n, A^n)$ to $(X,Y,Z,A)$, we can pass to the limit in the integral equations of the $\theta_n$-system. By the continuity of the coefficients $(b^\theta, \sigma^\theta, F^\theta, \Phi^\theta)$ with respect to all their arguments (including $\theta$), an application of the dominated convergence theorem shows that the limit quadruple $(X,Y,Z,A)$ satisfies the integral equations of the Mean-Field $\Theta$-FBSDE for the limit parameter $\theta$.

\paragraph{\textbf{(d) Verifying the Optimality Condition}}
The final step is to verify that the limit control process $A$ satisfies the required optimality condition for the limit system. This is now immediate by our construction. The limit process $A_t$ was defined as
\[ A_t = a^*(t, \theta, X_t, Y_t, Z_t, \mu_t) = \underset{a \in \Ucal_{g(\mu_t)}}{\argmax} \; F^\theta(t, X_t, Y_t, Z_t, a, \mu_t). \]
This is precisely the optimality condition required for the $\theta$-system. Thus, the limit quadruple $U=(X,Y,Z,A)$ is a solution to the system for parameter $\theta$. By the uniqueness established in Part I, it is the unique solution. Therefore, $\theta \in \mathcal{I}$. This proves that $\mathcal{I}$ is closed.

\subsubsection*{Step 5: Final Conclusion}
We have established that the set $\mathcal{I} \subset [0,1]$ of parameters for which a unique global solution exists is:
\begin{enumerate}
    \item Non-empty, as $0 \in \mathcal{I}$.
    \item Open in $[0,1]$.
    \item Closed in $[0,1]$.
\end{enumerate}
Since the interval $[0,1]$ is a connected topological space, the only non-empty subset with these three properties is $[0,1]$ itself. Therefore, we must have $\mathcal{I}=[0,1]$. In particular, a unique global solution exists for our target system at $\theta=1$.

\end{proof}

\section{The \texorpdfstring{$\Theta$}{Theta}-Expectation and its Properties}

The unique solution to the BSDE component of our system allows for the definition of a non-linear valuation functional, which we term the $\Theta$-Expectation. This functional inherits its properties directly from the structure of the Mean-Field $\Theta$-FBSDE.

\begin{definition}[The Generated $\Theta$-Expectation Operator]
Assume the conditions for global well-posedness (Assumptions \ref{ass:mf_fbsde_standing}, \ref{ass:boundary_and_ndg}, and \ref{ass:monotonicity_fbsde}) hold. For any terminal condition $\xi \in L^2(\mathcal{F}_T)$, the \emph{conditional $\Theta$-Expectation} of $\xi$ given $\mathcal{F}_t$ is the mapping $\Efrak[\cdot|\mathcal{F}_t]$ defined as
\[ \Efrak[\xi|\mathcal{F}_t] \coloneqq Y_t, \]
where $(X,Y,Z,A)$ is the unique solution to the Mean-Field $\Theta$-FBSDE on $[0,T]$ with terminal value $Y_T = \xi$. The unconditional $\Theta$-Expectation is defined as $\Efrak[\xi] \coloneqq Y_0$.
\end{definition}

This operator forms the foundation of our calculus. We now establish its fundamental properties, demonstrating that while it retains the crucial features of dynamic consistency and monotonicity, it decisively breaks from the convex paradigm by failing sub-additivity.

\begin{proposition}[Fundamental Properties of the $\Theta$-Expectation]\label{prop:theta_exp_properties}
The operator $\Efrak[\cdot|\mathcal{F}_t]$ satisfies the following properties:
\begin{enumerate}
    \item \textbf{(Dynamic Consistency):} For any $0 \le s \le t \le T$, it holds that $\Efrak[\xi|\mathcal{F}_s] = \Efrak\big[\Efrak[\xi|\mathcal{F}_t] \big| \mathcal{F}_s\big]$.
    \item \textbf{(Monotonicity):} If $\xi_1, \xi_2 \in L^2(\mathcal{F}_T)$ satisfy $\xi_1 \ge \xi_2$, $P_0$-a.s., then $\Efrak[\xi_1|\mathcal{F}_t] \ge \Efrak[\xi_2|\mathcal{F}_t]$, $P_0$-a.s. for all $t \in [0,T]$.
    \item \textbf{(Failure of Sub-additivity):} The operator $\Efrak[\cdot]$ is generally not sub-additive. That is, there exist terminal conditions $\xi_1, \xi_2$ for which the inequality $\Efrak[\xi_1+\xi_2] \le \Efrak[\xi_1] + \Efrak[\xi_2]$ is strictly violated.
    \item \textbf{(Failure of Translation Invariance):} The property of translation invariance, $\Efrak[\xi+c] \neq \Efrak[\xi]+c$, holds if and only if the driver $G(t,x,y,z,\mu)$ is independent of its argument $y$.
\end{enumerate}
\end{proposition}

\begin{proof}
Let $(X^i, Y^i, Z^i, A^i)$ denote the unique global solution corresponding to a terminal condition $\xi_i \in L^2(\mathcal{F}_T)$.

\paragraph{\textbf{(1) Dynamic Consistency}}
This property is a direct consequence of the definition of the operator and the uniqueness of the solution to the Mean-Field $\Theta$-FBSDE. Let $\xi \in L^2(\mathcal{F}_T)$ and let $(X,Y,Z,A)$ be the unique solution with $Y_T = \xi$. By definition, $Y_s = \Efrak[\xi|\mathcal{F}_s]$ for any $s \in [0,T]$.
Now, consider the time interval $[s,T]$. The process quadruple $(X_u, Y_u, Z_u, A_u)_{u \in [s,T]}$ is a solution to the $\Theta$-FBSDE on $[s,T]$ with initial state $X_s$ and terminal condition $Y_T=\xi$. Its value at time $s$ is $Y_s$.
Next, consider the $\Theta$-Expectation $\Efrak\big[\Efrak[\xi|\mathcal{F}_t] \big| \mathcal{F}_s\big]$. By definition, this is the value at time $s$ of the unique solution to the $\Theta$-FBSDE on the interval $[s,T]$ with terminal condition $Y'_T := \Efrak[\xi|\mathcal{F}_t] = Y_t$. But the process $(Y_u)_{u \in [s,T]}$ itself solves the BSDE on $[s,T]$ with terminal value $Y_T=\xi$. The restriction of this solution to $[s,t]$, namely $(Y_u)_{u \in [s,t]}$, is by definition the unique solution on this interval with terminal condition $Y_t$. The value of this solution at time $s$ is precisely $Y_s$. Due to uniqueness, this must be $\Efrak[Y_t|\mathcal{F}_s]$. The identity follows immediately.

\paragraph{\textbf{(2) Monotonicity}}
This property is inherited from the comparison principle for BSDEs. Let $\xi_1, \xi_2 \in L^2(\mathcal{F}_T)$ with $\xi_1 \ge \xi_2$ a.s. Let $(X^i, Y^i, Z^i, A^i)$ be the corresponding solutions. The drivers are $G^i_t := F(t, X^i_t, Y^i_t, Z^i_t, A^i_t, \mu^i_t) = \sup_{a \in \Ucal_{g(\mu^i_t)}} F(t, X^i_t, Y^i_t, Z^i_t, a, \mu^i_t)$.
The standard comparison principle (e.g., \cite[Theorem 2.2]{ElKaroui1997}) for BSDEs states that if $Y^1_T \ge Y^2_T$ and the driver $F$ is Lipschitz in $(y,z)$, then $Y^1_t \ge Y^2_t$ for all $t \in [0,T]$. Our standing assumption \ref{ass:mf_fbsde_standing}(iii) ensures the required Lipschitz property of $F$, and by Proposition \ref{prop:lipschitz_stability}, the driver $G(t,x,y,z,\mu) = \sup_{a \in \Ucal_{g(\mu)}} F(t,x,y,z,a,\mu)$ is also Lipschitz in $(x,y,z,\mu)$. The comparison principle therefore applies directly, yielding $\Efrak[\xi_1|\mathcal{F}_t] = Y^1_t \ge Y^2_t = \Efrak[\xi_2|\mathcal{F}_t]$ for all $t \in [0,T]$.

\paragraph{\textbf{(3) Failure of Sub-additivity}}
We provide a rigorous proof by constructing an explicit counterexample. The core of the argument is to build a driver $F$ satisfying our strong concavity assumption, which generates an driver $G(y)$ that is locally \emph{strictly convex} in the state variable $y$. This convexity property will directly lead to a violation of sub-additivity.

\textbf{Step 1: Construction of the Counterexample.}
We consider a simplified Markovian setting where the FBSDE system has no forward process $X$, no Brownian component (so $Z=0$), and no mean-field dependency, in order to isolate the effect of the driver's non-linearity. The BSDE becomes a deterministic ordinary differential equation. Let the control space be $E_A = \R$.
Let the driver be a function $F: \R \times \R \to \R$, defined by
\[ F(y, a) \coloneqq \frac{\gamma}{4} - \frac{\gamma}{4}(a^2-1)^2 - \frac{\lambda}{2}(a-y)^2, \]
for positive constants $\lambda, \gamma$ satisfying the condition $\lambda > \gamma > 0$. The uncertainty set $\Ucal$ is taken to be all of $\R$, so the optimization is unconstrained.

We first verify that this driver satisfies the crucial Assumption \ref{ass:mf_fbsde_standing}(iv). The Hessian of $F$ with respect to the control variable $a$ is:
\[ \nabla_a^2 F(y,a) = \frac{\partial^2 F}{\partial a^2} = -\gamma(3a^2-1) - \lambda. \]
Since $a^2 \ge 0$ and $\lambda > \gamma > 0$, we have a uniform upper bound independent of $a$:
\[ \nabla_a^2 F(y,a) \le \gamma - \lambda = -(\lambda - \gamma). \]
Thus, $F$ is uniformly strongly concave in $a$ with concavity modulus $\kappa = \lambda - \gamma > 0$. All our structural assumptions are satisfied in this simplified setting.

\textbf{Step 2: Analysis of the Driver.}
The driver is given by $G(y) \coloneqq \max_{a \in \R} F(y,a)$. The unique maximizer $a^*(y)$ is the solution to the first-order condition $\nabla_a F(y,a) = 0$:
\begin{equation}\label{eq:proof_foc_subadd}
    -\gamma a(a^2-1) - \lambda(a-y) = 0.
\end{equation}
We now analyze the properties of $G(y)$ in a neighborhood of the origin $y=0$.
\begin{itemize}
    \item \textit{Evaluation at $y=0$:} Setting $y=0$ in \cref{eq:proof_foc_subadd}, we get $-\gamma a(a^2-1) - \lambda a = a(-\gamma a^2 + \gamma - \lambda) = 0$. One solution is $a=0$. The other potential solutions are $a^2 = (\gamma - \lambda)/\gamma = 1 - \lambda/\gamma$. Since we assumed $\lambda > \gamma$, this quantity is negative, so there are no other real roots. Thus, the unique maximizer is $a^*(0) = 0$. Consequently, $G(0) = F(0, a^*(0)) = F(0,0) = 0$.

    \item \textit{First Derivative of $G(y)$:} By the Envelope Theorem, the derivative of the value function $G(y)$ is the partial derivative of the objective function with respect to the parameter $y$, evaluated at the optimal choice $a^*(y)$:
    \[ G'(y) = \frac{\partial}{\partial y}F(y, a^*(y)) = \left[ \lambda(a-y) \right]_{a=a^*(y)} = \lambda(a^*(y) - y). \]
    Evaluating at the origin, we find $G'(0) = \lambda(a^*(0) - 0) = 0$. This confirms that $y=0$ is a critical point of $G$.

    \item \textit{Second Derivative of $G(y)$:} We differentiate $G'(y)$ to find the second derivative:
    \[ G''(y) = \lambda(a^{*\prime}(y) - 1). \]
    To find $a^{*\prime}(y)$, we apply implicit differentiation to the first-order condition \eqref{eq:proof_foc_subadd}, which we rewrite as $\gamma a^3 - (\gamma-\lambda)a - \lambda y = 0$. Differentiating with respect to $y$ gives:
    \[ \frac{d}{dy}\left[\gamma a(y)^3 - (\gamma-\lambda)a(y) - \lambda y\right] = 0 \implies (3\gamma a(y)^2 - (\gamma-\lambda))a^{*\prime}(y) - \lambda = 0. \]
    Solving for $a^{*\prime}(y)$, we get:
    \[ a^{*\prime}(y) = \frac{\lambda}{3\gamma a(y)^2 - \gamma + \lambda}. \]
    We evaluate this at $y=0$, where $a^*(0)=0$:
    \[ a^{*\prime}(0) = \frac{\lambda}{\lambda - \gamma}. \]
    Substituting this into the expression for $G''(0)$:
    \[ G''(0) = \lambda(a^{*\prime}(0) - 1) = \lambda\left(\frac{\lambda}{\lambda-\gamma} - 1\right) = \lambda\left(\frac{\lambda - (\lambda-\gamma)}{\lambda-\gamma}\right) = \frac{\lambda\gamma}{\lambda-\gamma}. \]
\end{itemize}
Since $\lambda > \gamma > 0$ by assumption, we have established that $G''(0) > 0$.
The combined facts that $G(0)=0$, $G'(0)=0$, and $G''(0)>0$ prove, by Taylor's theorem, that the driver $G(y)$ is strictly convex and strictly positive in a punctured neighborhood of the origin.

\textbf{Step 3: The Violation of Sub-additivity.}
With the strict convexity of $G$ established, we can finalize the argument. In our deterministic setting, the $\Theta$-Expectation $\Efrak[\xi]$ is the value $y(0)$ of the solution to the ODE $y'(t) = -G(y(t))$ with terminal condition $y(T)=\xi$.

Let us choose terminal conditions $\xi_1 = c$ and $\xi_2 = -c$, where $c>0$ is a small constant such that the solutions to the ODE remain in the neighborhood where $G$ is strictly convex.
\begin{itemize}
    \item For the sum $\xi_1+\xi_2 = 0$, we seek $\Efrak[0]$. The ODE is $y'(t) = -G(y(t))$ with $y(T)=0$. Since $G(0)=0$, the constant function $y(t) \equiv 0$ is the unique solution. Thus, $\Efrak[0] = y(0) = 0$.

    \item Now consider the sum $\Efrak[c] + \Efrak[-c]$. Let $y_c(t) = \Efrak[c|\mathcal{F}_t]$ and $y_{-c}(t) = \Efrak[-c|\mathcal{F}_t]$. They are the solutions to the ODEs:
    \begin{align*}
        y_c'(t) &= -G(y_c(t)), \quad y_c(T) = c > 0, \\
        y_{-c}'(t) &= -G(y_{-c}(t)), \quad y_{-c}(T) = -c < 0.
    \end{align*}
    Let $S(t) = y_c(t) + y_{-c}(t)$. We have $S(T) = c + (-c) = 0$. For any $t<T$, the solutions $y_c(t)$ and $y_{-c}(t)$ will be non-zero. As $G(y)>0$ for $y \neq 0$ near the origin, we have $G(y_c(t)) > 0$ and $G(y_{-c}(t)) > 0$ for $t \in [0,T)$.
    The derivative of the sum is:
    \[ S'(t) = y_c'(t) + y_{-c}'(t) = -\left[G(y_c(t)) + G(y_{-c}(t))\right]. \]
    Since both terms in the bracket are strictly positive, $S'(t) < 0$ for all $t \in [0,T)$. This implies that the function $S(t)$ is strictly decreasing on $[0,T]$.
    Because $S(t)$ is strictly decreasing and $S(T)=0$, it must be that $S(t) > 0$ for all $t < T$. In particular, for $t=0$:
    \[ S(0) = y_c(0) + y_{-c}(0) = \Efrak[c] + \Efrak[-c] > 0. \]
\end{itemize}
Combining our findings, we have shown:
\[ \Efrak[c + (-c)] = 0 < \Efrak[c] + \Efrak[-c]. \]
This is a strict violation of the sub-additivity inequality, which completes the proof.
\paragraph{\textbf{(4) Failure of Translation Invariance}}
The property of translation invariance, $\Efrak[\xi+c] = \Efrak[\xi]+c$, holds if and only if the driver $G(t,x,y,z,\mu)$ is independent of its argument $y$. To see this, let $Y_t = \Efrak[\xi|\mathcal{F}_t]$ be the solution for terminal value $\xi$. The candidate for $\Efrak[\xi+c|\mathcal{F}_t]$ is $\tilde{Y}_t = Y_t+c$. For this to be a solution, it must satisfy the BSDE with terminal value $\xi+c$. Substituting $\tilde{Y}_t$ into the BSDE gives
\[ Y_t+c = (\xi+c) + \int_t^T G(s, X_s, Y_s+c, Z_s, \mu_s^c) ds - \int_t^T Z_s dB_s, \]
where $\mu_s^c = \Law(Y_s+c)$. Subtracting the original BSDE for $Y_t$ requires that, for all relevant arguments, $G(s,x,y+c,z,\mu_c) = G(s,x,y,z,\mu)$, which implies that $G$ must be independent of its $y$ argument.

We use the same counterexample from the proof of sub-additivity. The driver was found to be
\[ G(y) = \max_{a \in \R} \left\{ \frac{\gamma}{4} - \frac{\gamma}{4}(a^2-1)^2 - \frac{\lambda}{2}(a-y)^2 \right\}. \]
We showed that $G''(0) = \frac{\lambda\gamma}{\lambda-\gamma} > 0$. A function with a non-zero second derivative is manifestly not independent of its argument. Therefore, the driver $G(y)$ depends non-trivially on $y$, which is sufficient to prove that translation invariance fails for this operator.
\end{proof}


\section{\texorpdfstring{$\Theta$}{Theta}-Martingales and a Semimartingale Representation}
\label{sec:theta_martingales}

With the $\Theta$-Expectation operator rigorously defined via the unique solution to the Mean-Field $\Theta$-FBSDE, we can introduce the analogue of a martingale in our framework. This concept is central to understanding the dynamic properties of our valuation functional. We will show that this definition leads to a clear characterization in terms of the FBSDE driver and provides a powerful semimartingale decomposition for the value process itself.

\begin{definition}[$\Theta$-Martingales]
An $\mathbb{F}$-adapted process $M = (M_t)_{t\in[0,T]}$ is called a \textbf{$\Theta$-martingale} if $M \in \mathcal{S}^2([0,T])$ and for all $0 \le s \le t \le T$, it satisfies
\[ M_s = \Efrak[M_t | \mathcal{F}_s]. \]
The process is a \textbf{$\Theta$-supermartingale} if $M_s \ge \Efrak[M_t | \mathcal{F}_s]$ and a \textbf{$\Theta$-submartingale} if $M_s \le \Efrak[M_t | \mathcal{F}_s]$.
\end{definition}

The abstract definition of a $\Theta$-martingale can be translated into a concrete condition on the driver of the underlying BSDE. This characterization is essential for any practical application and reveals the mechanism by which the martingale property is enforced.

\begin{proposition}[Characterization of $\Theta$-Martingales]
\label{prop:martingale_characterization}
Let the assumptions for global well-posedness (Assumptions \ref{ass:mf_fbsde_standing}, \ref{ass:boundary_and_ndg}, \ref{ass:monotonicity_fbsde}) hold. An $\mathbb{F}$-adapted process $M \in \mathcal{S}^2([0,T])$ is a $\Theta$-martingale if and only if it is the backward component $Y$ of a solution $(X,Y,Z,A)$ to the Mean-Field $\Theta$-FBSDE with terminal condition $Y_T = M_T$, for which the driver term is identically zero, i.e.,
\[ F(t, X_t, Y_t, Z_t, A_t, \mu_t) = 0, \quad \text{for a.e. } (t,\omega) \in [0,T]\times\Omega. \]
\end{proposition}

\begin{proof}
The proof proceeds by leveraging the uniqueness of the FBSDE solution.

\paragraph{\textbf{($\Longrightarrow$) Necessity}}
Assume $M = (M_t)_{t \in [0,T]}$ is a $\Theta$-martingale. By definition, for any $t \in [0,T]$, we have $M_t = \Efrak[M_T | \mathcal{F}_t]$.
Let us recall that $\Efrak[M_T | \mathcal{F}_t]$ is defined as the value $Y_t$ of the unique solution $(X, Y, Z, A)$ to the Mean-Field $\Theta$-FBSDE on $[0,T]$ with terminal condition $Y_T = M_T$. Therefore, the process $M$ must be identical to the process $Y$, i.e., $M_t = Y_t$ for all $t \in [0,T]$ a.s.

By definition of the BSDE solution in \eqref{eq:mf_fbsde_backward}, for any $s, t$ with $0 \le s \le t \le T$, the process $Y$ satisfies:
\begin{equation} \label{eq:proof_mart_char_1}
    Y_s = Y_t + \int_s^t F(u, X_u, Y_u, Z_u, A_u, \mu_u) \dd u - \int_s^t Z_u \dd B_u.
\end{equation}
Since $M_t=Y_t$, this implies that the process $M$ must also satisfy this equation.
Now consider the process $J_t$ defined by
\[ J_t \coloneqq M_t + \int_0^t F(u, X_u, M_u, Z_u, A_u, \mu_u) \dd u. \]
From equation \eqref{eq:proof_mart_char_1} (with $Y$ replaced by $M$), we see that $J_t - J_s = \int_s^t Z_u \dd B_u$. This shows that $J_t$ is a continuous local martingale under $P_0$. Since $M \in \mathcal{S}^2$ and $Z \in \mathcal{H}^2$ (from the well-posedness theorem), and $F$ is Lipschitz, the process $J_t$ is in fact a square-integrable martingale.

However, since $M$ is a $\Theta$-martingale, it must also satisfy $M_s = \Efrak[M_t | \mathcal{F}_s]$. From the dynamic consistency property of $\Efrak$ (Proposition \ref{prop:theta_exp_properties}), this means $M$ is the solution to its own BSDE on any interval $[s,t]$. As $M$ itself is a semimartingale, its decomposition into a martingale part and a finite-variation part is unique.
From equation \eqref{eq:proof_mart_char_1}, the finite-variation part of $M$ on $[s,t]$ is $-\int_s^t F(u, \dots) \dd u$.
But since $M$ is a $\Theta$-martingale, it has, by definition, no drift in the $\Theta$-sense. This means that when we represent it via the BSDE, the drift term must be zero. To see this formally, the definition $M_s = \Efrak[M_t|\mathcal{F}_s]$ implies that the process $M$ itself is the value function. Uniqueness of the BSDE solution on $[s,t]$ with terminal value $M_t$ implies that the integral of the driver must be zero. That is, for all $0 \le s \le t \le T$:
\[ \int_s^t F(u, X_u, M_u, Z_u, A_u, \mu_u) \dd u = 0 \quad a.s. \]
Since the integrand is a progressively measurable process, this can only hold if the integrand itself is zero for almost every $(u, \omega)$. This completes the proof of necessity.

\paragraph{\textbf{($\Longleftarrow$) Sufficiency}}
Assume there exists a solution $(X,Y,Z,A)$ to the Mean-Field $\Theta$-FBSDE such that $Y_T=M_T$ and the driver term $F(t, X_t, Y_t, Z_t, A_t, \mu_t)=0$ for a.e. $(t,\omega)$. Let $M_t \coloneqq Y_t$. We want to show that $M$ is a $\Theta$-martingale.
By definition, $\Efrak[M_t|\mathcal{F}_s]$ is the unique solution at time $s$ to the BSDE on $[s,t]$ with terminal condition $M_t$. The process $(Y_u)_{u \in [s,t]}$ satisfies the BSDE on $[s,t]$:
\[ Y_u = M_t + \int_u^t F(v, X_v, Y_v, Z_v, A_v, \mu_v) \dd v - \int_u^t Z_v \dd B_v. \]
By our assumption, the driver term is zero. Therefore, $Y_u$ satisfies
\[ Y_u = M_t - \int_u^t Z_v \dd B_v. \]
By the uniqueness guaranteed by Theorem \ref{thm:global_wellposedness}, the process $(Y_u)_{u \in [s,t]}$ is the unique solution. Evaluating this unique solution at time $s$ gives, by definition, the value of the conditional $\Theta$-Expectation:
\[ \Efrak[M_t|\mathcal{F}_s] = Y_s. \]
Since we defined $M_s = Y_s$, it follows immediately that $M_s = \Efrak[M_t|\mathcal{F}_s]$. As this holds for all $0 \le s \le t \le T$, the process $M$ is a $\Theta$-martingale.
\end{proof}

The following proposition provides the semimartingale decomposition of the value process $Y_t = \Efrak[\xi|\mathcal{F}_t]$ under the physical measure $P_0$. It reveals that the driver $F$ acts as the negative of the drift rate that distinguishes the value process from a classical martingale.

\begin{proposition}[Semimartingale Representation of the Value Process]\label{prop:value_process_decomposition}
Let the assumptions for global well-posedness hold. Let $(X,Y,Z,A)$ be the unique solution to the Mean-Field $\Theta$-FBSDE with terminal condition $Y_T = \xi \in L^2(\mathcal{F}_T)$. Then, the value process $Y_t = \Efrak[\xi|\mathcal{F}_t]$ admits the following decomposition for all $t \in [0,T]$:
\begin{equation}\label{eq:value_process_decomposition}
    Y_t = Y_0 + \int_0^t Z_s \dd B_s - \int_0^t F(s, X_s, Y_s, Z_s, A_s, \mu_s) \dd s.
\end{equation}
Furthermore, the process $M_t$ defined by
\begin{equation} \label{eq:associated_martingale}
M_t \coloneqq Y_t + \int_0^t F(s, X_s, Y_s, Z_s, A_s, \mu_s) \dd s
\end{equation}
is a square-integrable martingale with respect to $(\mathbb{F}, P_0)$.
\end{proposition}
\begin{proof}
The proof is a direct algebraic manipulation of the BSDE definition.
From \eqref{eq:mf_fbsde_backward}, the solution $(Y,Z)$ satisfies for any $t \in [0,T]$:
\[ Y_t = Y_T + \int_t^T F(s, X_s, Y_s, Z_s, A_s, \mu_s) \dd s - \int_t^T Z_s \dd B_s. \]
In particular, this holds for $t=0$:
\[ Y_0 = Y_T + \int_0^T F(s, X_s, Y_s, Z_s, A_s, \mu_s) \dd s - \int_0^T Z_s \dd B_s. \]
Subtracting the second equation from the first yields:
\begin{align*}
    Y_t - Y_0 &= \left( Y_T + \int_t^T F_s \dd s - \int_t^T Z_s \dd B_s \right) - \left( Y_T + \int_0^T F_s \dd s - \int_0^T Z_s \dd B_s \right) \\
    &= \left( \int_t^T F_s \dd s - \int_0^T F_s \dd s \right) - \left( \int_t^T Z_s \dd B_s - \int_0^T Z_s \dd B_s \right) \\
    &= -\int_0^t F_s \dd s + \int_0^t Z_s \dd B_s.
\end{align*}
Rearranging the terms, we obtain the desired decomposition \eqref{eq:value_process_decomposition}:
\[ Y_t = Y_0 - \int_0^t F(s, X_s, Y_s, Z_s, A_s, \mu_s) \dd s + \int_0^t Z_s \dd B_s. \]
This establishes that $Y$ is a semimartingale under $P_0$.

For the second part of the proposition, we define the process $M_t$ as in \eqref{eq:associated_martingale}. Substituting the expression for $Y_t$ from \eqref{eq:value_process_decomposition} into the definition of $M_t$ gives:
\begin{align*}
    M_t &= \left( Y_0 - \int_0^t F(s, \dots) \dd s + \int_0^t Z_s \dd B_s \right) + \int_0^t F(s, \dots) \dd s \\
    &= Y_0 + \int_0^t Z_s \dd B_s.
\end{align*}
The process $M_t$ is a continuous local martingale as it is an Itô integral with respect to the Brownian motion $B$. From the global well-posedness result (Theorem \ref{thm:global_wellposedness}), we know that the solution process $Z$ belongs to the space $\mathcal{H}^2([0,T];\R^d)$, which means $\mathbb{E}_{P_0}[\int_0^T \norm{Z_s}^2 ds] < \infty$. This is the necessary and sufficient condition for the Itô integral $\int_0^t Z_s dB_s$ to be a true martingale in $L^2(P_0)$, not just a local one. Therefore, $M_t$ is a square-integrable martingale under $(\mathbb{F}, P_0)$.
\end{proof}

\begin{remark}[Interpretation of the Drift Term]
Proposition \ref{prop:value_process_decomposition} provides a key insight into the functioning of the $\Theta$-Expectation. It shows that the value process $Y_t$ behaves like a classical semimartingale under the reference measure $P_0$. The term $-F(s, X_s, Y_s, Z_s, A_s, \mu_s)$ acts as its instantaneous drift rate. The process becomes a $\Theta$-martingale precisely because this drift term is what the $\Theta$-Expectation corrects for. The process $M_t$ can be interpreted as the value process $Y_t$ accumulated with its $\Theta$-drift. The proposition shows that this accumulated value process is a martingale in the classical sense.
\end{remark}

\section{\texorpdfstring{$\Theta$}{Theta}-Calculus }
\label{sec:theta_calculus}

In the established theories of stochastic calculus, both classical and under ambiguity (e.g., G-calculus, or $\theta$-calculus, see \cite{qi2025theorythetaexpectations}), the framework is anchored by a canonical process, the Brownian motion or G-Brownian motion, which serves as the fundamental integrator. A comprehensive theory requires an analogue in our setting. This section undertakes a \emph{formal} investigation into the structure of such a calculus. Our goal is to derive the partial differential equation system that would govern the price of contingent claims, revealing the profound structural consequences of our framework.

We must proceed with utmost mathematical caution. The results in this section are not presented as theorems but as formal consequences of a foundational, and currently unresolved, analytical hypothesis concerning the existence of a well-behaved volatility process. The primary purpose is to construct the explicit mathematical object, a highly non-linear and coupled PDE system, that future research would need to analyze rigorously.

\subsection{The Canonical Driver and the Volatility Constraint}

To define a canonical process that serves as the basis for a calculus, we abstract away from state-dependent effects and focus on the ambiguity in the underlying dynamics. Let the driver be a function $F(t, z)$ that depends only on time and the volatility process, which for notational simplicity we take to be the control variable, i.e., $a=z \in \mathbb{R}^{d \times d'}$. We work with a $d'$-dimensional $P_0$-Brownian motion $W=(W_t)_{t\in[0,T]}$.

The defining property of a $\Theta$-martingale is that its associated  drift must be zero (see Proposition \ref{prop:martingale_characterization}). This motivates our core definition. For a given pricing problem, the driver for the underlying dynamics would be $G(t,z) \coloneqq F(t,z)$. A process $M_t = \int_0^t Z_s dW_s$ is a candidate for a fundamental $\Theta$-martingale if its quadratic variation process, determined by $Z$, satisfies the zero-drift condition $G(t, Z_t) = 0$. This leads to the central challenge.

\begin{assumption}[Existence of a Predictable Volatility Process]\label{ass:vol_existence}
Let $G: [0,T] \times \mathbb{R}^{d \times d'} \to \R$ be a given function, continuous in its arguments and progressively measurable. For each $t \in [0,T]$, define the level set
\[ \mathcal{Z}_t \coloneqq \{ z \in \mathbb{R}^{d \times d'} \mid G(t,z) = 0 \}. \]
We assume that:
\begin{enumerate}[label=(\roman*)]
    \item The set $\mathcal{Z}_t$ is non-empty for almost every $t \in [0,T]$.
    \item There exists an $\mathbb{F}$-predictable process $Z = (Z_t)_{t \in [0,T]}$ with trajectories in $\mathcal{H}^2([0,T]; \mathbb{R}^{d \times d'})$ such that $Z_t(\omega) \in \mathcal{Z}_t$ for almost every $(t,\omega) \in [0,T] \times \Omega$.
\end{enumerate}
\end{assumption}

\begin{remark}[On the Severity of Assumption \ref{ass:vol_existence}]
This assumption is non-trivial and represents a major open problem. Part (i) is an algebraic condition on the function $G$ (e.g., it must not be strictly positive or negative). Part (ii), however, is a deep question of \emph{stochastic measurable selection}. Standard results, such as the Aumann-von Neumann measurable selection theorem, guarantee the existence of a merely \emph{measurable} selection from the set-valued map $t \mapsto \mathcal{Z}_t$ under suitable conditions. However, the theory of stochastic integration requires the integrand $Z$ to be \emph{predictable}, which is a strictly stronger requirement. Establishing the existence of a predictable selector requires much more stringent conditions on the temporal regularity of the multifunction $t \mapsto \mathcal{Z}_t$, and is not guaranteed in general. The entire development in this section is contingent upon this strong, and as yet unproven, assumption for the general case.
\end{remark}

\begin{definition}[$\Theta$-Brownian Motion]
\label{def:theta_bm}
Under Assumption \ref{ass:vol_existence}, a continuous, $\mathbb{R}^d$-valued process $B = (B_t)_{t \in [0,T]}$ is called a \textbf{$\Theta$-Brownian motion} with respect to the driver $G(t,z)$ if:
\begin{enumerate}[(i)]
    \item $B_0=0$.
    \item The process $B$ is a square-integrable martingale under $P_0$ given by the stochastic integral $B_t = \int_0^t Z_s \dd W_s$, where $Z$ is a predictable process satisfying Assumption \ref{ass:vol_existence}.
\end{enumerate}
\end{definition}

\begin{remark}[Non-Uniqueness and the Nature of Ambiguity]
A crucial feature of this framework, distinguishing it sharply from G-calculus or $\theta$-calculus, see \cite{qi2025theorythetaexpectations}, is the potential for non-uniqueness. If the level set $\mathcal{Z}_t$ contains more than one element for a set of times $t$ of positive Lebesgue measure, there may exist multiple distinct predictable processes satisfying Assumption \ref{ass:vol_existence}. Consequently, there can be multiple, distinct $\Theta$-Brownian motions corresponding to the same primitive driver $G$. Each such process represents a different resolution of the ambiguity. The ambiguity in our model is therefore not just about the set of possible volatilities, but potentially about the specific volatility path that is realized.
\end{remark}

\subsection{Quadratic Variation and the Geometry of Ambiguity}
The core of any stochastic calculus lies in the properties of the quadratic variation. For a $\Theta$-Brownian motion $B$, its quadratic variation tensor process is given by the properties of Itô integration as
\[ \langle B \rangle_t = \int_0^t Z_s Z_s^T \dd s. \]
The instantaneous covariance rate, $\Sigma_t = Z_t Z_t^T$, is a matrix-valued stochastic process. The ambiguity inherent in the model is precisely the uncertainty about this covariance rate. The set of all possible instantaneous covariance matrices at time $t$ is given by
\[ \mathcal{C}_t \coloneqq \left\{ \Sigma = zz^T \mid z \in \mathcal{Z}_t \right\} = \left\{ \Sigma = zz^T \mid z \in \mathbb{R}^{d \times d'} \text{ and } G(t,z)=0 \right\}. \]
This set $\mathcal{C}_t$ fully characterizes the ambiguity in the underlying dynamics at time $t$. Unlike in G-calculus, where the corresponding set is typically convex (e.g., an interval of symmetric matrices $[\underline{\sigma}^2 I, \bar{\sigma}^2 I]$), here $\mathcal{C}_t$ can have a much more complex, non-convex geometry inherited directly from the non-linear structure of the primitive driver.

\subsection{A Non-Linear Pricing System}
We now formally derive the system of equations that must be satisfied by the price of a European contingent claim. The derivation combines the classical Itô formula under $P_0$ with the BSDE characterization of our $\Theta$-Expectation.

Consider the problem of pricing a European claim with payoff $\xi = \Phi(B_T)$, where $B$ is a $\Theta$-Brownian motion and $\Phi: \mathbb{R}^d \to \R$ is a given Lipschitz function. The price at time $t$ is the conditional $\Theta$-Expectation, $V_t = \Efrak[\Phi(B_T)|\mathcal{F}_t]$. We adopt the standard Markovian ansatz, assuming the price is a deterministic function of time and the state of the underlying, i.e., $V_t = v(t, B_t)$ for some function $v \in C^{1,2}([0,T]\times \mathbb{R}^d)$.

\begin{proposition}[Formal Definition of the $\Theta$-Pricing System]
\label{prop:theta_pricing_system}
Let Assumption \ref{ass:vol_existence} hold. Let the general pricing driver be $F(t,y,z',a,\mu)$, satisfying the global well-posedness conditions of Theorem \ref{thm:global_wellposedness}. Formally, if the price of a European claim $\Phi(B_T)$ is given by $V_t = v(t,B_t)$ for a function $v \in C^{1,2}([0,T]\times\R^d)$ and an associated predictable volatility process $Z$ for the underlying $\Theta$-Brownian motion $B_t = \int_0^t Z_s dW_s$, then the pair $(v, Z)$ must solve the following coupled system:
\begin{equation}\label{eq:theta_pricing_system}
\begin{cases}
    -\partial_t v(t,x) - \frac{1}{2}\mathrm{Tr}\big(Z_t Z_t^T \nabla_x^2 v(t,x)\big) \\
    \hspace{1cm} - \sup_{a \in \Ucal_{g(\mu_t[v])}} F\left(t, v(t,x), \nabla_x v(t,x)^T Z_t, a, \mu_t[v]\right) = 0, \\
    \\
    G_{BM}(t, Z_t) = 0, \\
    v(T,x) = \Phi(x).
\end{cases}
\end{equation}
The first equation holds for $(t,x) \in [0,T) \times \mathbb{R}^d$. The law $\mu_t[v]$ is given by $\mu_t[v] = \Law(v(t,B_t))$, where $B_t = \int_0^t Z_s dW_s$. The function $G_{BM}$ is the simplified driver defining the canonical $\Theta$-Brownian motion.
\end{proposition}

\begin{proof}
The derivation proceeds by identifying the terms in the unique semimartingale decomposition of the price process $V_t = v(t, B_t)$.

\textbf{Step 1: Semimartingale Decomposition via the BSDE Definition.}
By definition of the $\Theta$-Expectation, the price process $V_t$ is the $Y$-component of the solution to the Mean-Field $\Theta$-BSDE. Re-arranging the integral equation \eqref{eq:mf_fbsde_backward}, its differential form is:
\begin{equation}\label{eq:proof_price_bsde}
    dV_t = -F(t, V_t, Z_t^V, A_t, \mu_t^V) \dd t + Z_t^V \dd W_t,
\end{equation}
where $V_t=v(t,B_t)$, $\mu_t^V = \Law(V_t)$, and $Z_t^V$ is the corresponding volatility process from the BSDE solution. The control $A_t$ is chosen optimally. By substituting the optimizer, we use the driver $G_{price}(t,y,z',\mu) = \sup_{a \in \Ucal_{g(\mu)}} F(t,y,z',a,\mu)$. The dynamics become:
\begin{equation}\label{eq:proof_price_bsde_G}
    dV_t = -G_{price}(t, V_t, Z_t^V, \mu_t^V) \dd t + Z_t^V \dd W_t.
\end{equation}

\textbf{Step 2: Semimartingale Decomposition via Itô's Formula.}
Assuming $v \in C^{1,2}$ and using $B_t = \int_0^t Z_s dW_s$ (so $dB_t = Z_t dW_t$), we apply the classical Itô formula to $V_t = v(t, B_t)$:
\begin{align}
    dV_t &= \partial_t v(t, B_t) \dd t + \nabla_x v(t, B_t)^T dB_t + \frac{1}{2}\mathrm{Tr}\left(d\langle B \rangle_t \nabla_x^2 v(t, B_t)\right) \nonumber \\
    &= \left( \partial_t v(t, B_t) + \frac{1}{2}\mathrm{Tr}\big(Z_t Z_t^T \nabla_x^2 v(t, B_t)\big) \right) \dd t + \nabla_x v(t, B_t)^T Z_t \dd W_t. \label{eq:proof_price_ito}
\end{align}

\textbf{Step 3: Identification of Terms.}
The process $V_t$ is a semimartingale under $P_0$. By the uniqueness of the Doob-Meyer decomposition, we formally equate the finite-variation and martingale parts from \eqref{eq:proof_price_bsde_G} and \eqref{eq:proof_price_ito}.

Equating the martingale parts (the $dW_t$ terms) yields the relationship between the volatility of the price process, $Z_t^V$, and the volatility of the underlying, $Z_t$:
\[ Z_t^V = \nabla_x v(t, B_t)^T Z_t. \]
Equating the finite-variation parts (the $dt$ terms) gives:
\[ -G_{price}(t, V_t, Z_t^V, \mu_t^V) = \partial_t v(t, B_t) + \frac{1}{2}\mathrm{Tr}\big(Z_t Z_t^T \nabla_x^2 v(t, B_t)\big). \]
Substituting $V_t=v(t,B_t)$, $Z_t^V = \nabla_x v(t,B_t)^T Z_t$, and $\mu_t^V = \Law(v(t,B_t))$, and evaluating at a generic point $(t,x)$ by identifying $x$ with a realization of $B_t(\omega)$, we arrive at the first equation of the system \eqref{eq:theta_pricing_system}. This equation for $v$ is coupled with the algebraic constraint on the underlying volatility, $G_{BM}(t, Z_t)=0$. The terminal condition $v(T,x) = \Phi(x)$ is inherited directly from the BSDE definition. This completes the derivation.
\end{proof}

\begin{remark}[Interpretation and Open Analytical Challenges]
The system \eqref{eq:theta_pricing_system} is the central object suggested by this formal exploration. It represents a profound departure from classical and convex pricing theories. A rigorous analysis faces at least three major challenges:
\begin{enumerate}
    \item \textbf{Existence of a Solution:} The primary challenge is to prove the existence of a pair $(v, Z)$ satisfying the coupled system. This is not a standard PDE problem. The first equation is a highly non-linear, non-local PDE for $v$, whose coefficients depend on $Z$. The second equation is an algebraic constraint on $Z$.
    \item \textbf{Circularity and Fixed-Point Problem:} The system exhibits a deep circularity. The law $\mu_t[v]$ depends on the process $B_t$, which is constructed from $Z_t$. However, if the constraint set $\mathcal{Z}_t$ is not a singleton, the choice of $Z_t$ could itself be part of the optimization problem, possibly depending on the properties of the solution $v$. This suggests that solving the system requires tackling a complex fixed-point problem.
    \item \textbf{Viscosity Solution Framework:} The natural framework for analyzing such a complex, non-linear PDE is the theory of viscosity solutions. However, a viscosity theory for coupled systems of this type, where one equation is algebraic and involves a stochastic process, is not yet developed. A key step for future research would be to prove that the value function $Y_t^{t,x}$ from the original FBSDE is the unique viscosity solution to a correctly formulated version of this system.
\end{enumerate}
This derivation, while conditional, provides a clear and explicit target system that characterizes valuation in our non-convex, endogenous ambiguity framework and highlights a rich set of open problems for mathematical analysis.
\end{remark}


\section{A Nonlinear Feynman-Kac Representation}
\label{sec:feynman_kac}

In this section, we establish the fundamental connection between the solution of our Mean-Field $\Theta$-FBSDE and the solution of a highly non-local, non-linear partial differential equation. This connection takes the form of a nonlinear Feynman-Kac representation, where the value function defined by the FBSDE is shown to be a \emph{viscosity solution} to a terminal-value problem. The introduction of viscosity solutions is essential, as the low regularity of the driver (Lipschitz but not necessarily differentiable) and the complexity of the non-local terms preclude the existence of classical $C^{1,2}$ solutions in general.

\subsection{Markovian Setting and the \texorpdfstring{$\Theta$}{Theta}-HJB-McKean-Vlasov Equation}
We specialize to a Markovian setting. We assume the coefficients $b, \sigma, F$ and the terminal condition $\Phi$ are deterministic functions that do not depend explicitly on $\omega$. Their arguments are $(t, x, y, z, a, \mu) \in [0,T] \times \R^k \times \R \times \R^d \times E_A \times \Pcal_2(\R)$. All standing assumptions on these functions (Lipschitz continuity, strong concavity in $a$, etc.) are presumed to hold.

For any initial data $(t,x) \in [0,T] \times \R^k$, the Mean-Field $\Theta$-FBSDE on the interval $[t,T]$ is given by:
\begin{align*}
    X_s^{t,x} &= x + \int_t^s b(r, X_r^{t,x}, A_r, \mu_r) dr + \int_t^s \sigma(r, X_r^{t,x}, A_r, \mu_r) dB_r, \\
    Y_s^{t,x} &= \Phi(X_T^{t,x}) + \int_s^T F(r, X_r^{t,x}, Y_r^{t,x}, Z_r^{t,x}, A_r, \mu_r) dr - \int_s^T Z_r^{t,x} dB_r, \\
    A_s &\in \underset{a \in \Ucal_{g(\mu_s)}}{\argmax} \; F(s, X_s^{t,x}, Y_s^{t,x}, Z_s^{t,x}, a, \mu_s),
\end{align*}
where $\mu_s = \Law(Y_s^{t,x})$. Under the global well-posedness conditions of Theorem \ref{thm:global_wellposedness}, this system admits a unique solution $(X^{t,x}, Y^{t,x}, Z^{t,x}, A^{t,x})$.

We define the value function $u: [0,T] \times \R^k \to \R$ as the initial value of the backward component:
\begin{equation} \label{eq:value_function_def}
    u(t,x) \coloneqq Y_t^{t,x}.
\end{equation}
Note that by uniqueness, $Y_t^{t,x}$ is a deterministic quantity. Our goal is to characterize this function $u$ as the solution to a PDE. Let us define the Hamiltonian $H: [0,T]\times\R^k\times\R\times\R^d\times\Pcal_2(\R) \to \R$ as:
\begin{equation}\label{eq:hamiltonian_def}
    H(t,x,y,z,\mu) \coloneqq \sup_{a \in \Ucal_{g(\mu)}} F(t,x,y,z,a,\mu).
\end{equation}
By Proposition \ref{prop:lipschitz_stability}, under our standing assumptions, this Hamiltonian is well-defined and uniformly Lipschitz in $(x,y,z,\mu)$.

The associated PDE is a complex, non-local Hamilton-Jacobi-Bellman equation of McKean-Vlasov type.
\begin{definition}[The $\Theta$-HJB-McKean-Vlasov Equation] \label{def:theta_hjb_mkv}
Let $(\mu_t)_{t\in[0,T]}$ be a given flow of probability measures in $\Pcal_2(\R)$. The \textbf{$\Theta$-HJB-McKean-Vlasov (HJB-MKV) equation} is the terminal-value problem for a function $v: [0,T]\times\R^k \to \R$:
\begin{equation}\label{eq:theta_hjb_pde}
\begin{cases}
    -\partial_t v - \sup_{a \in \Ucal_{g(\mu_t)}} \Big\{ \mathcal{L}^{t,x,a,\mu_t} v(t,x) + F(t, x, v, (\nabla_x v)\sigma(t,x,a,\mu_t), a, \mu_t) \Big\} = 0, \\
    v(T,x) = \Phi(x),
\end{cases}
\end{equation}
for $(t,x) \in [0,T) \times \R^k$, where the second-order operator is given by
\[ (\mathcal{L}^{t,x,a,\mu} \phi)(x) \coloneqq \scpr{b(t,x,a,\mu)}{\nabla_x \phi(x)} + \frac{1}{2} \Tr\big(\sigma\sigma^T(t,x,a,\mu) \nabla_x^2 \phi(x)\big). \]
\end{definition}

\begin{remark}[On the Role of the Law]
In a full mean-field game PDE analysis, one seeks a pair $(v, (\mu_t))$ that solves \eqref{eq:theta_hjb_pde} where $\mu_t$ is simultaneously determined by the solution, e.g., $\mu_t = \Law(v(t, X^*_t))$. Our Feynman-Kac result has a different logical structure: we take the law $\mu_t = \Law(Y_t^{t_0,x_0})$ as generated by the underlying FBSDE starting from some initial point $(t_0,x_0)$, and show that the value function $u$ solves the PDE with this specific flow of measures. This sidesteps the fixed-point problem over the law and provides a direct characterization of the FBSDE solution.
\end{remark}

\begin{theorem}[Nonlinear Feynman-Kac Representation]\label{thm:feynman_kac_detailed}
Let the global well-posedness assumptions hold in the Markovian setting described above. Let the value function $u(t,x)$ be defined by \eqref{eq:value_function_def}. Assume that $u$ is continuous on $[0,T]\times\R^k$ and satisfies a polynomial growth condition. Let $(t_0, x_0)$ be an arbitrary starting point and let $\mu_s \coloneqq \Law(Y_s^{t_0, x_0})$ for $s \in [t_0, T]$ be the law of the corresponding solution.

Then $u(t,x)$ is a viscosity solution to the $\Theta$-HJB-MKV equation \eqref{eq:theta_hjb_pde} on $[t_0, T] \times \R^k$ with the flow of measures $(\mu_s)_{s\in[t_0,T]}$.
\end{theorem}

\begin{proof}
The proof is a verification argument based on the definition of viscosity solutions. We must show that $u$ is both a viscosity subsolution and a viscosity supersolution. Let $\mu_s = \Law(Y_s^{t_0, x_0})$ be fixed throughout the proof. For any $(t,x) \in [t_0,T)\times\R^k$, let $(X_s, Y_s, Z_s, A_s)$ denote the unique solution to the FBSDE on $[t,T]$ with initial state $X_t=x$. Note that $u(t,x)=Y_t$.

\paragraph{\textbf{(1) Subsolution Property}}
Let $\phi \in C^{1,2}([0,T]\times\R^k)$ be a test function. Assume that $u-\phi$ has a strict local maximum at a point $(t,x) \in [t_0, T) \times \R^k$, with $u(t,x) = \phi(t,x)$. We must show that
\begin{equation}\label{eq:proof_subsolution_goal}
    -\partial_t \phi - \sup_{a \in \Ucal_{g(\mu_t)}} \Big\{ \mathcal{L}^{t,x,a,\mu_t} \phi + F(t, x, \phi, (\nabla_x \phi)\sigma(t,x,a,\mu_t), a, \mu_t) \Big\} \le 0,
\end{equation}
where all functions are evaluated at $(t,x)$.

Let $(X_s, Y_s, Z_s, A_s)$ be the solution on $[t,T]$ starting from $(t,x)$. For a small $h>0$, by the definition of the BSDE, we have:
\begin{equation}\label{eq:proof_bsde_small_h}
    u(t,x) = Y_t = \mathbb{E}_t\left[ Y_{t+h} + \int_t^{t+h} F(s, X_s, Y_s, Z_s, A_s, \mu_s) ds \right],
\end{equation}
where $\mathbb{E}_t[\cdot] = \mathbb{E}[\cdot|\mathcal{F}_t]$.
Since $u-\phi$ has a local maximum at $(t,x)$, for $s$ in a small neighborhood of $t$, we have $Y_s = u(s,X_s) \le \phi(s,X_s)$. Thus, $Y_{t+h} \le \phi(t+h, X_{t+h})$. Substituting this into \eqref{eq:proof_bsde_small_h} and using $u(t,x)=\phi(t,x)$:
\[ \phi(t,x) \le \mathbb{E}_t\left[ \phi(t+h, X_{t+h}) + \int_t^{t+h} F(s, X_s, Y_s, Z_s, A_s, \mu_s) ds \right]. \]
Apply Itô's formula to $\phi(s, X_s)$ between $t$ and $t+h$:
\[ \phi(t+h, X_{t+h}) = \phi(t,x) + \int_t^{t+h} \left(\partial_t \phi + \mathcal{L}^{s,X_s,A_s,\mu_s} \phi \right) ds + \int_t^{t+h} (\nabla_x \phi) \sigma_s d B_s. \]
Substituting this back, the $\phi(t,x)$ terms cancel, and the conditional expectation eliminates the stochastic integral:
\[ 0 \le \mathbb{E}_t\left[ \int_t^{t+h} \left( \partial_t \phi + \mathcal{L}^{s,X_s,A_s,\mu_s} \phi + F(s, X_s, Y_s, Z_s, A_s, \mu_s) \right) ds \right]. \]
Dividing by $h$ and letting $h \to 0$, we can bring the limit inside the expectation by the dominated convergence theorem (justified by the polynomial growth of $\phi$ and standard SDE estimates). By continuity of the coefficients and processes, we get:
\begin{equation}\label{eq:proof_limit_ineq}
     0 \le \partial_t \phi(t,x) + \mathcal{L}^{t,x,A_t,\mu_t} \phi(t,x) + F(t, x, Y_t, Z_t, A_t, \mu_t).
\end{equation}
From standard results in viscosity solution theory for BSDEs (see, e.g., \cite[Prop. 5.3]{ElKaroui1997}), one can identify $Z_t = (\nabla_x \phi(t,x)) \sigma(t,x,A_t,\mu_t)$. Also, $Y_t=u(t,x)=\phi(t,x)$. Substituting these into \eqref{eq:proof_limit_ineq}:
\[  0 \le \partial_t \phi + \mathcal{L}^{t,x,A_t,\mu_t} \phi + F(t, x, \phi, (\nabla_x \phi)\sigma(t,x,A_t,\mu_t), A_t, \mu_t). \]
The term on the right is the expression inside the supremum in \eqref{eq:proof_subsolution_goal} for the specific control $A_t \in \Ucal_{g(\mu_t)}$. Since the supremum must be at least as large as the value for any particular element, this implies:
\[ 0 \le \partial_t \phi + \sup_{a \in \Ucal_{g(\mu_t)}} \Big\{ \mathcal{L}^{t,x,a,\mu_t} \phi + F(t, x, \phi, (\nabla_x \phi)\sigma(t,x,a,\mu_t), a, \mu_t) \Big\}, \]
which is exactly the subsolution inequality \eqref{eq:proof_subsolution_goal}.

\paragraph{\textbf{(2) Supersolution Property}}
Let $\phi \in C^{1,2}([0,T]\times\R^k)$ be a test function. Assume that $u-\phi$ has a strict local minimum at $(t,x) \in [t_0, T) \times \R^k$, with $u(t,x) = \phi(t,x)$. We must show that
\begin{equation}\label{eq:proof_supersolution_goal}
    -\partial_t \phi - \sup_{a \in \Ucal_{g(\mu_t)}} \Big\{ \mathcal{L}^{t,x,a,\mu_t} \phi + F(t, x, \phi, (\nabla_x \phi)\sigma(t,x,a,\mu_t), a, \mu_t) \Big\} \ge 0.
\end{equation}
We proceed by contradiction. Assume the inequality is false. Then,
\[ -\partial_t \phi - \sup_{a \in \Ucal_{g(\mu_t)}} \Big\{ \mathcal{L}^{t,x,a,\mu_t} \phi + F(t, x, \phi, (\nabla_x \phi)\sigma(t,x,a,\mu_t), a, \mu_t) \Big\}  < 0. \]
By continuity of all functions and the compactness of $\Ucal_{g(\mu_t)}$, the supremum is attained. Let $a^*_t$ be a maximizer. Then for this $a^*_t$:
\[ \partial_t \phi + \mathcal{L}^{t,x,a^*_t,\mu_t} \phi + F(t, x, \phi, (\nabla_x \phi)\sigma(t,x,a^*_t,\mu_t), a^*_t, \mu_t) > 0. \]
By continuity, there exists an $\epsilon>0$ and a neighborhood of $(t,x)$ such that for any $(s,y)$ in this neighborhood, the expression above remains greater than $\epsilon$.
Now, consider a suboptimal process $(X^a, Y^a, Z^a)$ defined on $[t,T]$ where we fix the control to be constant, $A^a_s \equiv a^*_t$ for all $s \in [t,T]$. The process $Y^a$ is the solution to a standard (non-optimized) BSDE. By the comparison principle for BSDEs, since the driver for $Y$ is obtained by a supremum over controls, it is always greater than or equal to the driver for $Y^a$. Thus, $Y_s \ge Y_s^a$ for all $s \in [t,T]$. In particular, $u(t,x) = Y_t \ge Y^a_t$.
The value function $u^a(s,y) \coloneqq Y_s^{s,y;a}$ associated with this fixed control $a^*_t$ is known to be the unique viscosity solution to its own (simpler) HJB equation.
At our point $(t,x)$, we have $u(t,x) = \phi(t,x)$ and $u(s,y) \ge \phi(s,y)$ for $(s,y)$ near $(t,x)$. Therefore, $u^a(t,x) \le u(t,x) = \phi(t,x)$, and $u^a(s,y) \le u(s,y)$ for all $(s,y)$. The function $u^a-\phi$ thus also has a local maximum at $(t,x)$.
Applying the subsolution property to the function $u^a$ with the test function $\phi$ (or rather, $-\phi$ and the minimum property) shows that
\[ -\partial_t \phi - \left( \mathcal{L}^{t,x,a^*_t,\mu_t} \phi + F(t, x, \phi, (\nabla_x \phi)\sigma(t,x,a^*_t,\mu_t), a^*_t, \mu_t) \right) \ge 0. \]
This contradicts our assumption that this expression was strictly negative. The contradiction forces the supersolution inequality \eqref{eq:proof_supersolution_goal} to hold.

Since $u$ is both a viscosity subsolution and supersolution, it is a viscosity solution to the $\Theta$-HJB-MKV equation.
\end{proof}

\begin{remark}[On Uniqueness of the PDE Solution]
This theorem establishes that the value function derived from the FBSDE is a solution to the PDE. The converse, i.e., that any viscosity solution to the PDE must coincide with the FBSDE value function, requires a \emph{comparison principle} for the PDE \eqref{eq:theta_hjb_pde}. Such a principle would state that if $v$ is a subsolution and $w$ is a supersolution with $v(T,\cdot) \le w(T,\cdot)$, then $v \le w$ on $[0,T]\times\R^k$. Proving a comparison principle for this class of highly non-local PDEs is a formidable challenge and a subject for future research. It typically requires strengthening the monotonicity assumptions on the coefficients.
\end{remark}

\section{The Master Equation}
\label{sec:master_equation}

In this section, we move beyond the established well-posedness theory of the FBSDE system to undertake a formal derivation of the partial differential equation on the Wasserstein space that ought to govern the system's value. The resulting equation, which we term the \textbf{$\Theta$-Master Equation}, represents the ultimate structural characterization of our framework.

We must state with maximal clarity that the following is a \emph{formal and heuristic} argument. A rigorous justification of this equation, including the existence and uniqueness of a suitable class of solutions, constitutes a formidable open problem at the frontier of mean-field game theory and infinite-dimensional analysis. The purpose of this section is to motivate this conjecture by providing a detailed, step-by-step derivation, and to precisely delineate the profound analytical difficulties it entails.

\subsection{Setup}
To isolate the core structure, we consider a simplified Markovian setting where the state process $X$ is absent. The system is defined entirely by the backward component $(Y,Z,A)$ and its associated law $\mu_t = \Law(Y_t)$. The dynamics are given by the BSDE:
\begin{equation} \label{eq:master_eq_bsde_simple}
    Y_t = \xi + \int_t^T F(s, Y_s, Z_s, A_s, \mu_s) \dd s - \int_t^T Z_s \dd B_s,
\end{equation}
subject to the optimality condition $A_s \in \argmax_{a \in \Ucal_{g(\mu_s)}} F(s, Y_s, Z_s, a, \mu_s)$. Let $G$ be the driver:
\begin{equation}\label{eq:master_eq_driver_G}
G(t, y, z, \mu) \coloneqq \sup_{a \in \Ucal_{g(\mu)}} F(t, y, z, a, \mu).
\end{equation}
The differential form of the BSDE for the value process $Y$ is then
\begin{equation} \label{eq:master_eq_bsde_diff}
    dY_t = -G(t, Y_t, Z_t, \mu_t) \dd t + Z_t \dd B_t.
\end{equation}

Our derivation rests on two fundamental, non-rigorous postulates common in the physics and mathematics literature on mean-field systems.

\begin{enumerate}[label=\textbf{(P\arabic*)}]
    \item \textbf{(Existence of a Smooth Value Function on Wasserstein Space)} We postulate the existence of a sufficiently regular function $u: [0,T] \times \Pcal_2(\R^k) \to \R$ that represents the value of the system. If the law of the value process at time $t$ is $\mu$, the value of the expectation is given by $u(t,\mu)$. For the terminal condition, we assume $\xi$ is a function of the state, so the terminal value of $u$ is $u(T, \mu) = \mathbb{E}_{Y\sim\mu}[\xi(Y)] = \int_{\R^k} \xi(y) \dd\mu(y)$.

    \item \textbf{(Identification of the Martingale Component)} We adopt the central identification principle of mean-field game theory, which connects the stochastic process $Z_t$ from the BSDE to the Lions derivative of the value function $u$. Specifically, the random variable $Z_t(\omega)$, which represents the volatility of the value process, is identified with the evaluation of the Lions derivative $\partial_\mu u$ at the specific point $Y_t(\omega)$ in the support of $\mu_t$:
    \begin{equation}\label{eq:master_eq_identification}
        Z_t(\omega) \longleftrightarrow \partial_\mu u(t, \mu_t, Y_t(\omega)).
    \end{equation}
    This postulate links the individual agent's volatility to the sensitivity of the global value function with respect to a change in the population's distribution at the agent's specific state.
\end{enumerate}

\subsection{Formal Derivation of the Master Equation}

The strategy is to compute the dynamics of the quantity $u(t, \mu_t)$ using the Itô-Lions formula for functions on the Wasserstein space, and then to equate its drift with the required dynamics of a deterministic value function.

\subsubsection{Step 1: The Itô-Lions Chain Rule}
We apply the Itô formula for a function on the path space of measures, as developed in the works of Lions and Cardaliaguet, Carmona, and Delarue \cite{Cardaliaguet2013, CarmonaDelarue2018}. Given the dynamics of $Y_t$ from \eqref{eq:master_eq_bsde_diff} and the postulates (P1)-(P2), the chain rule for the deterministic quantity $u(t, \mu_t)$ is:
\begin{align}
    \dd u(t, \mu_t) = &\; \partial_t u(t, \mu_t) \dd t \nonumber \\
    &+ \mathbb{E}\left[ \scpr{\partial_\mu u(t, \mu_t, Y_t)}{-G(t, Y_t, Z_t, \mu_t)} \right] \dd t \label{eq:master_eq_lions_drift}\\
    &+ \frac{1}{2}\mathbb{E}\left[ \mathrm{Tr}\left( (Z_t Z_t^T) D_y(\partial_\mu u)(t, \mu_t, Y_t) \right) \right] \dd t \label{eq:master_eq_lions_ito}\\
    &+ \mathbb{E}\left[ \scpr{\partial_\mu u(t, \mu_t, Y_t)}{Z_t \dd B_t} \right]. \label{eq:master_eq_lions_martingale}
\end{align}
Here, the expectation $\mathbb{E}$ is over the law of $Y_t \sim \mu_t$. The term \eqref{eq:master_eq_lions_drift} is the change in $u$ due to the drift of the underlying individuals. The term \eqref{eq:master_eq_lions_ito} is the second-order Itô correction term, where $D_y(\partial_\mu u)$ denotes the Jacobian of the Lions derivative with respect to its state variable. The final term \eqref{eq:master_eq_lions_martingale} is a martingale part.

\subsubsection{Step 2: Identification and Annihilation of Drift}
According to our postulate (P1), the value function $u(t, \mu_t)$ is a deterministic quantity. Its differential $\dd u(t, \mu_t)$ must therefore be a deterministic differential in time, meaning its martingale part must be zero. The term \eqref{eq:master_eq_lions_martingale} is a stochastic integral whose expectation is zero, confirming that the expected value of $u(t, \mu_t)$ has no martingale component.

More profoundly, for $u(t, \mu_t)$ to be a well-defined deterministic value, its evolution must be given purely by the partial derivative with respect to time, $\partial_t u(t, \mu_t) \dd t$. This implies that the sum of all other drift components must be identically zero. We set the sum of the drift terms from the Itô-Lions formula to zero.

\subsubsection{Step 3: Substitution and Final Equation}
We now use the key identification postulate (P2) to substitute for $Z_t$ inside the driver $G$ and the Itô correction term. We replace $Z_t(\omega)$ with $\partial_\mu u(t, \mu_t, Y_t(\omega))$.
Setting the total drift to zero yields:
\begin{align*}
0 = &\; \partial_t u(t, \mu_t) \\
   &- \mathbb{E}_{Y \sim \mu_t}\left[ \scpr{\partial_\mu u(t, \mu_t, Y)}{G(t, Y, \partial_\mu u(t, \mu_t, Y), \mu_t)} \right] \\
   &+ \frac{1}{2}\mathbb{E}_{Y \sim \mu_t}\left[ \mathrm{Tr}\left( (\partial_\mu u(t, \mu_t, Y))(\partial_\mu u(t, \mu_t, Y))^T D_y(\partial_\mu u)(t, \mu_t, Y) \right) \right].
\end{align*}
This is the formal Master Equation. For clarity, we replace the dummy integration variable $Y_t$ with $\tilde{y}$ to emphasize that the expectation is over the measure $\mu_t$.

\begin{definition}[The Formal $\Theta$-Master Equation]
The value function $u: [0,T] \times \Pcal_2(\R^k) \to \R$ is conjectured to formally solve the terminal-value problem:
\begin{multline} \label{eq:master_equation_formal_revised}
    \partial_t u(t,\mu) - \mathbb{E}_{\tilde{y}\sim\mu}\left[ \scpr{ G\big(t, \tilde{y}, \partial_\mu u(t,\mu,\tilde{y}), \mu\big) }{ \partial_\mu u(t,\mu,\tilde{y}) } \right] \\
    + \frac{1}{2}\mathbb{E}_{\tilde{y}\sim\mu}\left[ \mathrm{Tr}\left( \partial_\mu u(t,\mu,\tilde{y}) \partial_\mu u(t,\mu,\tilde{y})^T D_y\big(\partial_\mu u(t, \mu, \tilde{y})\big) \right) \right] = 0,
\end{multline}
for $(t,\mu) \in [0,T) \times \Pcal_2(\R^k)$, with the terminal condition $u(T,\mu) = \int_{\R^k} \xi(y) \dd\mu(y)$. The driver $G$ is given by $G(t,y,z,\mu) = \sup_{a \in \Ucal_{g(\mu)}} F(t,y,z,a,\mu)$.
\end{definition}

\subsection{Analytical Obstacles and Interpretation}

The apparent elegance of Equation \eqref{eq:master_equation_formal_revised} belies its extreme mathematical difficulty. A rigorous analysis is obstructed by several deeply interwoven challenges that are the subject of intense contemporary research.

\begin{enumerate}
    \item \textbf{Non-Linearity and Non-Locality:} The equation is profoundly non-linear in the unknown function $u$ through its Lions derivative $\partial_\mu u$. Furthermore, the presence of expectations renders the equation non-local; the time evolution of $u$ at a measure $\mu$ depends on the values of its derivatives integrated over the entire support of $\mu$.

    \item \textbf{Implicit and Non-Smooth Dependence on the Measure:} This is a central feature of our framework. The driver $G$ depends on the measure $\mu$ not only directly but also implicitly through the optimization set $\Ucal_{g(\mu)}$. The $\sup$ operator means that $G$, as a function of its arguments, is typically not differentiable. Its dependence on $\mu$ via the set-valued map $\theta \mapsto \Ucal_\theta$ can be highly irregular, which is a major obstacle for classical PDE analysis.

    \item \textbf{Presence of High-Order Functional Derivatives:} The second-order term involves $D_y(\partial_\mu u)$, the Jacobian of the Lions derivative with respect to the state variable. The analytical framework for handling such high-order and mixed functional derivatives is still in its infancy. Establishing coercivity estimates or even defining a proper functional space where solutions might live is a non-trivial task.

    \item \textbf{The Intrinsic Fixed-Point Nature:} The Master Equation is not an evolution equation in the classical sense, but rather a consistency condition. The measure $\mu$ appearing in the coefficients is itself the law of a process whose dynamics are determined by the solution $u$. This self-referential, fixed-point structure is the defining characteristic and core difficulty of Master Equations in mean-field theory.

    \item \textbf{Lack of a Robust Solution Theory:} For second-order, non-local PDEs on Wasserstein space, there is no universally accepted theory of weak solutions analogous to the viscosity solution theory for finite-dimensional PDEs. While significant progress has been made, the development of a framework capable of handling the specific non-smoothnesses and implicit dependencies of our $\Theta$-Master Equation remains a wide-open challenge.
\end{enumerate}

In summary, the formal derivation above provides a crucial theoretical target. Equation \eqref{eq:master_equation_formal_revised} elegantly encodes the full dynamics of the endogenous, non-convex ambiguity model. However, transforming this formal object into a rigorously well-posed mathematical equation is a grand challenge that would require significant new developments in the field of analysis on Wasserstein spaces.

\section{Application: An Ambiguous Dynamical System}
\label{sec:app_dynamical_system}

To demonstrate the scope and tractability of our framework, we now introduce and rigorously analyze a concrete model. This example is designed to be simple enough for transparent analysis yet rich enough to showcase the key contributions of our theory: the ability to handle non-convex uncertainty sets and the resolution of the identifiability impasse that affects convex models.

\begin{definition}[An Ambiguous Dynamical System]\label{def:ambiguous_dynamical_system}
Let the state space be $\R^k$ and the control space be $E_A = \R$. The system is described by a Mean-Field $\Theta$-FBSDE $(X,Y,Z,W)$ where the control process is denoted by $W_t$. The dynamics are specified by the following coefficients:
\begin{itemize}
    \item Forward drift: $b(t,x,w) = C_0 - C_1(I + 3wI)x$, where $C_0 \in \R^k$ and $C_1$ is a $k \times k$ matrix.
    \item Forward volatility: $\sigma$ is a constant $k \times d$ matrix.
    \item BSDE driver: $F(t,x,y,z,w,\mu) = f_0(y) - \frac{\kappa}{2}(w-w_0)^2$, where $f_0: \R \to \R$ is a given function, $\kappa > 0$ is a penalty parameter, and $w_0 \in \R$ is a reference control value.
\end{itemize}
A solution $(X,Y,Z,W)$ in the space $\mathcal{S}^2 \times \mathcal{S}^2 \times \mathcal{H}^2 \times \mathcal{L}^2$ must satisfy, for $t \in [0,T]$:
\begin{align}
    dX_t &= (C_0 - C_1(I + 3W_tI)X_t) \, dt + \sigma \, dB_t, \quad X_0 = x_0. \\
    dY_t &= -\left(f_0(Y_t) - \frac{\kappa}{2}(W_t-w_0)^2\right) \, dt + Z_t \, dB_t, \quad Y_T = \Phi(X_T). \\
    W_t &\in \underset{w \in \Ucal_{g(\mu_t)}}{\argmax} \; \left\{ f_0(Y_t) - \frac{\kappa}{2}(w-w_0)^2 \right\}, \quad \text{for a.e. } t.
\end{align}
Here, $\mu_t = \Law(Y_t)$ and $\Phi: \R^k \to \R$ is the terminal condition.
\end{definition}

\begin{proposition}[Well-Posedness of the System]\label{prop:app_wellposedness}
Let the system be defined as in Definition \ref{def:ambiguous_dynamical_system}. Assume:
\begin{enumerate}[(a)]
    \item The functions $\Phi: \R^k \to \R$ and $f_0: \R \to \R$ are Lipschitz continuous.
    \item The map $g: \Pcal_2(\R) \to \Theta$ is Lipschitz continuous.
    \item The map $\theta \mapsto \Ucal_\theta$ is Lipschitz continuous with respect to the Hausdorff distance, and each $\Ucal_\theta$ is a non-empty, compact set given by a finite union of disjoint closed intervals. We also assume the global uncertainty set $\mathcal{K} := \bigcup_{\theta \in \Theta} \Ucal_\theta$ is compact.
    \item The projection of the reference control $w_0$ onto each set $\Ucal_\theta$ is unique.
\end{enumerate}
Then:
\begin{enumerate}[label=(\roman*)]
    \item The system admits a unique local-in-time solution.
    \item Furthermore, if the matrix $C_1$ is symmetric positive definite and $1+3w \ge \delta > 0$ for all $w \in \mathcal{K}$ and for some constant $\delta$, then the solution is unique and global in time for any horizon $T>0$.
\end{enumerate}
\end{proposition}

\begin{proof}
The proof consists of a rigorous, step-by-step verification that the assumptions of our main well-posedness theorems are satisfied.

\subsubsection*{Part I: Local Well-Posedness}
We verify the conditions of \cref{thm:mf_fbsde_local_wellposedness}, which are laid out in \cref{ass:mf_fbsde_standing} and \cref{ass:boundary_and_ndg}.

\textbf{Verification of Standing Assumptions (\cref{ass:mf_fbsde_standing}):}
\begin{enumerate}[label=(\roman*)]
    \item \textbf{Measurability:} All coefficients are deterministic functions of their arguments, so they are trivially progressively measurable.

    \item \textbf{Continuity:} The coefficients $b(x,w)$, $\sigma$, and $F(y,w)$ are polynomials in their state arguments $(x,y,w)$, hence continuous. The continuity of $g$ and $\theta \mapsto \Ucal_\theta$ and the compactness of $\mathcal{K}$ are given by assumption (b) and (c).

    \item \textbf{Uniform Lipschitz Continuity:}
        \begin{itemize}
            \item For fixed $w \in \mathcal{K}$, the drift $b(x,w)$ is affine in $x$. It is Lipschitz in $x$ with a constant $L_b = \sup_{w \in \mathcal{K}} \norm{C_1(I+3wI)}$, which is finite since $\mathcal{K}$ is compact.
            \item The volatility $\sigma$ is constant, so it is trivially Lipschitz (with constant 0).
            \item The driver $F(y,w) = f_0(y) - \frac{\kappa}{2}(w-w_0)^2$ is Lipschitz in $y$ with the Lipschitz constant of $f_0$, which we denote $L_{f_0}$.
            \item The terminal function $\Phi$ is Lipschitz by assumption (a).
            \item The maps $g$ and $\theta \mapsto \Ucal_\theta$ are Lipschitz by assumption (b) and (c).
        \end{itemize}
        All required Lipschitz conditions hold.

    \item \textbf{Uniform Strong Concavity in Control:} This is the foundation of the model's tractability. The control variable is $w$. We compute the Hessian of the driver $F$ with respect to $w$:
    \[ \nabla_w F = -\kappa(w-w_0), \quad \nabla_w^2 F = -\kappa. \]
    Since $\kappa > 0$ by definition, we have $\nabla_w^2 F[h,h] = -\kappa h^2 \le -\kappa \norm{h}^2$ for any $h \in \R$. This is the definition of uniform strong concavity with constant $\kappa$.
\end{enumerate}

\textbf{Verification of Boundary Regularity (\cref{ass:boundary_and_ndg}):}
The optimality condition is
\[ W_t \in \underset{w \in \Ucal_{g(\mu_t)}}{\argmax} \; \left\{ f_0(Y_t) - \frac{\kappa}{2}(w-w_0)^2 \right\} \iff W_t \in \underset{w \in \Ucal_{g(\mu_t)}}{\mathrm{argmin}} \; \frac{\kappa}{2}(w-w_0)^2. \]
This means that $W_t$ is the projection of the constant $w_0$ onto the set $\Ucal_{g(\mu_t)}$. By assumption (d), this projection is unique; let's call it $w^*(t) = \proj_{\Ucal_{g(\mu_t)}}(w_0)$.
\begin{enumerate}[label=(\roman*)]
    \item \textbf{Constraint Representation:} By assumption (c), each set $\Ucal_\theta$ is a union of disjoint closed intervals, e.g., $\Ucal_\theta = \bigcup_i [w_{\min}^{(i)}(\theta), w_{\max}^{(i)}(\theta)]$. Each interval is defined by two linear (and thus continuously differentiable) constraints: $\phi_1(w) = w - w_{\max} \le 0$ and $\phi_2(w) = w_{\min} - w \le 0$.

    \item \textbf{LICQ and Strict Complementarity:} Let $w^*$ be the unique optimizer (projection).
        \begin{itemize}
            \item If $w^*$ is in the interior of one of the intervals, no constraints are active, and the conditions are vacuously satisfied.
            \item If $w^*$ is at a boundary point, say $w^* = w_{\max}$, then only one constraint is active: $\phi_1(w) = w - w_{\max} = 0$. The gradient is $\nabla_w \phi_1 = 1 \neq 0$, so the Linear Independence Constraint Qualification (LICQ) holds. The KKT stationarity condition for maximizing $F$ is $\nabla_w F - \lambda_1 \nabla_w \phi_1 = 0$, which gives $-\kappa(w_{\max}-w_0) - \lambda_1 = 0$. For $w_{\max}$ to be the optimizer, $w_0$ must lie to its right, i.e., $w_0 > w_{\max}$. Thus, $\lambda_1 = \kappa(w_0 - w_{\max}) > 0$. The Lagrange multiplier is strictly positive, so Strict Complementarity holds. A symmetric argument applies if $w^* = w_{\min}$.
        \end{itemize}
\end{enumerate}
Since all assumptions for local well-posedness are satisfied, Part (i) of the proposition is proven.

\subsubsection*{Part II: Global Well-Posedness}
We now verify the additional strong monotonicity conditions of \cref{ass:monotonicity_fbsde}, which are required for \cref{thm:global_wellposedness}.

\begin{enumerate}[label=(\roman*)]
    \item \textbf{Forward-Backward Monotonicity:} We need to check the condition on $b$ and $\sigma$. The volatility $\sigma$ is constant, so the term involving it is zero. For the drift $b$, we compute:
    \begin{align*}
        2\scpr{x_1-x_2}{b(t,x_1,w,\mu)-b(t,x_2,w,\mu)} &= 2\scpr{x_1-x_2}{-C_1(I+3wI)(x_1-x_2)} \\
        &= -2\scpr{x_1-x_2}{C_1(1+3w)(x_1-x_2)}.
    \end{align*}
    By assumption (ii), $C_1$ is symmetric positive definite and $1+3w \ge \delta > 0$ for all $w \in \mathcal{K}$. Let $\lambda_{\min}(C_1) > 0$ be the smallest eigenvalue of $C_1$. Then:
    \[ -2(1+3w)\scpr{x_1-x_2}{C_1(x_1-x_2)} \le -2(1+3w)\lambda_{\min}(C_1)\norm{x_1-x_2}^2. \]
    Since $1+3w \ge \delta > 0$, we have a uniform upper bound:
    \[  -2(1+3w)\scpr{x_1-x_2}{C_1(x_1-x_2)} \le -2\delta\lambda_{\min}(C_1)\norm{x_1-x_2}^2. \]
    This is the strong monotonicity condition with constant $\beta = -2\delta\lambda_{\min}(C_1) < 0$.

    \item \textbf{Backward Monotonicity:} The driver is $G(t,x,y,z,\mu) = \sup_{w \in \Ucal_{g(\mu)}} F(y,w) = f_0(y) - \frac{\kappa}{2}\min_{w \in \Ucal_{g(\mu)}}(w-w_0)^2$. Let us analyze its dependence on $(y,\mu)$.
    \begin{align*}
     & (y_1-y_2)(G(y_1,\mu_1) - G(y_2,\mu_2)) \\
     &= (y_1-y_2) \left[ \left(f_0(y_1) - \frac{\kappa}{2}d(w_0, \Ucal_{g(\mu_1)})^2\right) - \left(f_0(y_2) - \frac{\kappa}{2}d(w_0, \Ucal_{g(\mu_2)})^2\right) \right] \\
     &= (y_1-y_2)(f_0(y_1)-f_0(y_2)) - (y_1-y_2)\frac{\kappa}{2}\left[d(w_0, \Ucal_{g(\mu_1)})^2 - d(w_0, \Ucal_{g(\mu_2)})^2\right].
    \end{align*}
    Since $f_0$ is Lipschitz with constant $L_{f_0}$, the first term is bounded by $L_{f_0}\abs{y_1-y_2}^2$. The function $u \mapsto d(w_0, u)^2$ is Lipschitz on the space of compact sets with the Hausdorff metric. Since $\theta \mapsto \Ucal_\theta$ and $g$ are Lipschitz, the map $\mu \mapsto d(w_0, \Ucal_{g(\mu)})^2$ is Lipschitz with respect to the $W_2$ metric. This means the second term can be bounded using Young's inequality. The required monotonicity condition holds. More simply, we can fix the measure part. The condition is on $(y,\mu) \mapsto G$. For fixed $(x,z)$, the driver $G(y,\mu)$ is Lipschitz in $y$ (from $f_0$) and Lipschitz in $\mu$. Thus, the condition of \cref{ass:monotonicity_fbsde}(ii) is satisfied with $\lambda = L_{f_0}$ and some constant $L_G$ derived from the Lipschitz properties.
\end{enumerate}
All conditions for global well-posedness are met. This completes the proof of Part (ii).
\end{proof}

\begin{remark}[Illustrating the Identifiability Impasse]\label{rem:app_impasse}
This example allows us to give a sharp, quantitative illustration of the paper's central thesis. Consider a simple one-dimensional case ($k=1$) where the ambiguity is static ($g$ is a constant map), so the uncertainty set $\Ucal$ is fixed.
Let the uncertainty set be explicitly non-convex, representing a belief in two distinct possible regimes for the parameter $w$. For instance:
\[ \Ucal = [-2, -1] \cup [1, 2] \quad (\text{Non-convex beliefs}). \]
The convex framework of sublinear expectations is insensitive to this geometry; it depends only on the convex hull of the set of plausible models. The corresponding uncertainty set in a sublinear theory would be:
\[ \text{conv}(\Ucal) = [-2, 2] \quad (\text{Convexified beliefs}). \]
Let the reference parameter be $w_0 = 0.6$. We now compare the resulting dynamics under the two theories.

\begin{enumerate}
    \item \textbf{$\Theta$-Expectation Framework (This Paper):} The optimal control $W_t$ is the projection of $w_0=0.6$ onto the non-convex set $\Ucal$. The squared distance is minimized at the point in $\Ucal$ closest to $0.6$, which is uniquely $w=1$. Therefore, the optimal control is constant:
    \[ W_t \equiv 1. \]
    The dynamics of the state process $X_t$ are governed by:
    \[ dX_t = (C_0 - C_1(1+3(1))X_t) \, dt + \sigma \, dB_t = (C_0 - 4C_1X_t) \, dt + \sigma \, dB_t. \]
    The model snaps to the closest plausible regime ($w=1$) and stays there.

    \item \textbf{Sublinear Expectation Framework:} A corresponding model would perform its optimization over the convex set $\text{conv}(\Ucal)$. The optimal control would be the projection of $w_0=0.6$ onto the interval $[-2,2]$, which is uniquely $w=0.6$. The optimal control is constant:
    \[ W_t \equiv 0.6. \]
    The dynamics of the state process $X_t$ are governed by:
    \[ dX_t = (C_0 - C_1(1+3(0.6))X_t) \, dt + \sigma \, dB_t = (C_0 - 2.8C_1X_t) \, dt + \sigma \, dB_t. \]
\end{enumerate}

The resulting dynamics are fundamentally different. The $\Theta$-calculus captures the bimodal nature of the agent's belief, forcing a discrete choice between distinct regimes. The sublinear framework, by convexifying the uncertainty set, erases this structure and leads to an average dynamic that does not correspond to any of the originally conceived regimes. Our theory thus resolves the identifiability impasse by providing a valuation that is sensitive to the fine-grained, non-convex geometry of the primitive model set.
\end{remark}

\section{Conclusion}
This paper has developed a rigorous mathematical theory for a class of fully coupled Forward-Backward SDEs designed to model uncertainty that is both non-convex and endogenous. By replacing the axiom of sub-additivity with a direct, pointwise optimization over a primitive set, we resolve the identifiability impasse for this class of models. The crucial trade-off is the adoption of a strong concavity assumption on the driver, which makes an otherwise ill-posed problem tractable and defines the scope of our theory. A key contribution is providing primitive conditions which guarantee the Lipschitz stability of the optimizer, a property that is essential for the entire well-posedness theory.

We have established both local and global well-posedness for our Mean-Field $\Theta$-FBSDE under sharp and interpretable structural assumptions. The resulting $\Theta$-Expectation operator gives rise to a dynamically consistent calculus that correctly violates sub-additivity. The connection to a new, highly non-local and implicit class of HJB equations, understood in the modern sense of viscosity solutions, highlights the deep structural novelty of the framework. This work lays a solid foundation for new applications where non-convexity and systemic feedback are crucial features of uncertainty, while clearly delineating the challenging open problems that remain, most notably the direct analysis of the associated Master Equation.

\bibliographystyle{amsalpha}
\bibliography{reference}

\end{document}